\documentclass[11pt,preprint]{article}

\usepackage{amsmath, latexsym, amsfonts, amssymb, amsthm}
\usepackage{graphicx, color,hyperref,dsfont,epsfig,caption,wrapfig,subfig}
\usepackage{pstricks}
\usepackage{movie15}
\hypersetup{
    linktoc=page,
    linkcolor=red,          
    citecolor=blue,        
    filecolor=blue,      
    urlcolor=cyan,
    colorlinks=true           
}

\numberwithin{equation}{section}

\newtheorem{theorem}{Theorem}[section]
\newtheorem{lemma}[theorem]{Lemma}
\newtheorem{proposition}[theorem]{Proposition}
\newtheorem{corollary}[theorem]{Corollary}

\newcounter{conj}
\newtheorem{conjecture}[conj]{Conjecture}

\theoremstyle{definition}

 \title{\Large{The Fyodorov-Bouchaud formula and Liouville conformal field theory}}

 \author{ Guillaume Remy\footnote{ D\'epartement de math\'ematiques et applications, \'Ecole normale sup\'erieure, CNRS, PSL Research University, 75005 Paris, France. Research supported in part by the ANR grant Liouville (ANR-15-CE40-0013).}  }

  \date{\vspace{-5ex}}
  
\begin{document}

  \maketitle

\begin{abstract}

In a remarkable paper in 2008, Fyodorov and Bouchaud conjectured an exact formula for the density of the total mass of (sub-critical) Gaussian multiplicative chaos (GMC) associated to the Gaussian free field (GFF) on the unit circle \cite{FyBo}. In this paper we will give a proof of this formula. In the mathematical literature this is the first occurrence of an explicit probability density for the total mass of a GMC measure. The key observation of our proof is that the negative moments of the total mass of GMC determine its law and are equal to one-point correlation functions of Liouville conformal field theory in the disk defined by Huang, Rhodes and Vargas \cite{Disk}. The rest of the proof then consists in implementing  rigorously the framework of conformal field theory (BPZ equations for degenerate field insertions) in a probabilistic setting to compute the negative moments. Finally we will discuss applications to random matrix theory, asymptotics of the maximum of the GFF and tail expansions of GMC.  
  \end{abstract}


\tableofcontents
  \vspace{0.1cm}
  
\pagebreak

\section{Introduction and main result}

Starting from a Gaussian free field (GFF) one can by standard regularization techniques define the associated Gaussian multiplicative chaos (GMC) measure whose density is formally given by the exponential of the GFF. The theory of GMC goes back to Kahane's 1985 paper \cite{Kah} and has grown into an important field within probability theory and mathematical physics with applications to 3d turbulence, mathematical finance, extreme values of log-correlated processes, disordered systems, random geometry and 2d quantum gravity. See for instance \cite{review} for a review.

In this paper we will be concerned with the last application and more precisely with the link between GMC and the correlation functions of Liouville conformal field theory (LCFT). It is this connection uncovered in 2014 in \cite{Sphere} that enables us to understand the integrability of GMC measures and perform exact computations. The very recent proof of the DOZZ formula \cite{DOZZ1, DOZZ2} can be seen as the first integrability result on fractional moments of GMC measures while our Theorem \ref{FyBo} is the first result that gives an explicit probability density for the total mass of a GMC measure.

We will now introduce the framework of our paper. Let $X$ be a GFF on the unit disk $\mathbb{D}$ with covariance given for $x,y \in \mathbb{D}$ by:\footnote{This is the GFF with Neumann boundary conditions also called the free boundary conditions, see \cite{Disk}.}
\begin{equation}\label{Cov}
 \mathbb{E}[ X(x) X(y)] = \ln \frac{1}{\vert x - y \vert \vert 1 - x \overline{y}\vert }.  
\end{equation} 
In the case of two points $e^{i \theta}$ and $ e^{i \theta'}$ on the unit circle $\partial \mathbb{D}$, this simply reduces to:\footnote{This normalization is different from the $\ln \frac{1}{\vert e^{i \theta} - e^{i \theta'} \vert} $ usually found in the literature.}
\begin{equation}\label{Cov2}
\mathbb{E}[ X(e^{i \theta}) X(e^{i \theta'})] = 2 \ln \frac{1}{\vert e^{i \theta} - e^{i \theta'} \vert}.
\end{equation}
In this setting and for all $ \gamma \geq 0$, the GMC measure on the unit circle is constructed as the following limit in probability in the sense of weak convergence of measures,
\begin{equation}\label{defintro} 
e^{\frac{\gamma}{2} X(e^{i \theta})} d\theta : = \underset{\epsilon \rightarrow 0}{\lim} \: e^{\frac{\gamma}{2} X_{\epsilon}(e^{i \theta}) - \frac{\gamma^2}{8} \mathbb{E} [X_{\epsilon}(e^{i \theta})^2] } d \theta,
\end{equation}
where $d \theta$ is the Lebesgue measure on $[0,2\pi]$ and $X_\epsilon$ is a reasonable cut-off approximation of $X$ which converges to $X$ as $\epsilon$ goes to $0$. More precisely for any continuous test function $f : \partial \mathbb{D} \mapsto \mathbb{R}$, the following holds in probability:
\begin{equation} 
\int_0^{2 \pi} e^{\frac{\gamma}{2} X(e^{i \theta})} f(e^{i \theta}) d\theta = \underset{\epsilon \rightarrow 0}{\lim} \: \int_0^{2 \pi} e^{\frac{\gamma}{2} X_{\epsilon}(e^{i \theta}) - \frac{\gamma^2}{8} \mathbb{E} [X_{\epsilon}(e^{i \theta})^2] } f(e^{i \theta}) d \theta.
\end{equation}
See for instance Berestycki \cite{Ber} for an elegant proof of this convergence. It goes back to Kahane \cite{Kah} that  the measure $e^{\frac{\gamma}{2} X(e^{i \theta})} d\theta $ defined by \eqref{defintro}   is different from $0$ if and only if $\gamma \in [0,2)$. In the sequel, we will always work with $\gamma \in (0,2)$ (with the exception of section \ref{app1} where we will discuss the limit $\gamma \rightarrow 2$). We now introduce the main quantity of interest of our paper, the partition function of the theory, for $\gamma \in (0,2)$:
\begin{equation}
Y_\gamma = \frac{1}{2 \pi} \int_0^{2 \pi} e^{ \frac{\gamma}{2} X(e^{i \theta})} d \theta.
\end{equation}
Recall the following classical fact on the existence of moments for  GMC (see the reviews by Rhodes-Vargas \cite{review, Houches} for instance), for $p \in \mathbb{R}$:
\begin{equation}\label{existence_moments}
\mathbb{E}[ Y_\gamma^p ] < + \infty \: \: \Longleftrightarrow  \: \:  p< \frac{4}{\gamma^2}.
\end{equation}
In 2008 Fyodorov and Bouchaud \cite{FyBo} conjectured an exact formula for the density of $Y_\gamma$ (see also \cite{FLeR, Ostro1} for more conjectures for GMC on the unit interval $[0,1]$\footnote{This case corresponds to $X$ log-correlated on $[0,1]$ meaning that \eqref{Cov2} would become $\mathbb{E}[X(x)X(y)] = 2 \ln \frac{1}{\vert x - y \vert} $.}). Their conjecture is based on the computation of the integer moments of $Y_{\gamma}$ and a clever observation. If $X_\epsilon$ is a reasonable cut-off approximation of $X$ then for all $p$ \emph{nonnegative integer} such that $p <\frac{4}{\gamma^2}$ one gets by Fubini (interchanging  $\int_0^{2 \pi} $ and $\mathbb{E}[.]$):
\begin{align*} 
 \mathbb{E}\left[ \left(\frac{1}{2 \pi}\int_0^{2\pi} e^{\frac{\gamma}{2} X_{\epsilon}(e^{i \theta}) - \frac{\gamma^2}{8} \mathbb{E} [X_{\epsilon}(e^{i \theta})^2] } d \theta \right)^p\right] & = \frac{1}{(2 \pi)^p} \int_{ [0,2 \pi]^p}   \mathbb{E}\left[ \prod_{i=1}^p e^{ \frac{\gamma}{2} X_\epsilon(e^{i \theta_i})-\frac{\gamma^2}{8} \mathbb{E} [X_{\epsilon}(e^{i \theta_i})^2] }  \right] d\theta_1 \dots d \theta_p   \\
& = \frac{1}{(2 \pi)^p} \int_{ [0,2 \pi]^p} e^{ \frac{\gamma^2}{4} \sum_{ i < j} \mathbb{E}[X_\epsilon(e^{i \theta_i}) X_\epsilon(e^{i \theta_j}) ] } d\theta_1 \dots d\theta_p 
\end{align*}
By taking the limit in the above computation as $\epsilon$ goes to $0$ one gets:
 \begin{align*} 
\mathbb{E}[ Y_\gamma^p] & =   \frac{1}{(2 \pi)^p} \int_{ [0,2 \pi]^p} e^{ \frac{\gamma^2}{4} \sum_{ i < j} \mathbb{E}[X(e^{i \theta_i}) X(e^{i \theta_j}) ] } d\theta_1 \dots d\theta_p \\
&  = \frac{1}{(2 \pi)^p} \int_{ [0,2 \pi]^p} \prod_{i<j} \frac{1}{\vert e^{i \theta_i} - e^{i \theta_j} \vert^{ \frac{\gamma^2}{2}} } d\theta_1 \dots d\theta_p.   
\end{align*}

 The main observation of  Fyodorov and Bouchaud \cite{FyBo} is that the last integral above is a circular variant of the famous Selberg integral, the so-called Morris integral, and its value is explicitly known as $\frac{\Gamma(1 - p \frac{\gamma^2}{4})}{\Gamma(1 - \frac{\gamma^2}{4})^p}$. Hence, one gets for $p$ nonnegative integer such that $p <\frac{4}{\gamma^2}$: 
\begin{equation}\label{keyobs}
\mathbb{E}[ Y_\gamma^p] = \frac{\Gamma(1 - p \frac{\gamma^2}{4})}{\Gamma(1 - \frac{\gamma^2}{4})^p}.
\end{equation}
Fyodorov and Bouchaud then conjectured that the identity \eqref{keyobs}  remains valid if $p$ is any real number such that $p< \frac{4}{\gamma^2}$. Though the conjecture is reasonable, one should notice that it is far from obvious. Indeed both sides of \eqref{keyobs}  are analytic functions of $p$ equal on the finite set $\lbrace 0, 1, \cdots, \lfloor \frac{4}{\gamma^2} \rfloor \rbrace $ (where $ \lfloor . \rfloor$ denotes integer part) and this does not guarantee that they are equal on their domain of definition. The main result of this paper is precisely to prove this point:
\begin{theorem} (Fyodorov-Bouchaud formula) \label{FyBo}
Let $\gamma \in (0,2)$. For all real $p$ such that $p < \frac{4}{\gamma^2}$ the following identity holds:
\begin{equation}\label{maintheorem} 
\mathbb{E}[ Y_\gamma ^p ] =  \frac{\Gamma(1 - p \frac{\gamma^2}{4})}{\Gamma(1 - \frac{\gamma^2}{4})^p}.  
\end{equation}
Equivalently, we can state that $Y_{\gamma} $ follows the law
\begin{equation}\label{law_Yg}
Y_{\gamma} \sim \frac{1}{\Gamma(1 - \frac{\gamma^2}{4})} \mathcal{E}(1)^{- \frac{\gamma^2}{4}}
\end{equation}
where $\mathcal{E}(1) $ is an exponential law of parameter $1$.
\end{theorem}
From \eqref{law_Yg} the variable $Y_\gamma$ thus has an explicit density given by,
\begin{equation}\label{densityY}
 f_{Y_\gamma}(y) =  \frac{4 \beta }{\gamma^2} ( \beta y )^{ - \frac{4}{\gamma^2}-1} e^{ -( \beta y)^{-\frac{4}{\gamma^2} }} \mathbf{1}_{[0, \infty[ }(y),
\end{equation}
where we have set $\beta = \Gamma(1 - \frac{\gamma^2}{4}) $. This density illustrates two universal behaviours of GMC measures. First, the probability for $Y_{\gamma}$ to be large is governed by the term $ ( \beta y )^{ - \frac{4}{\gamma^2}-1}$. The power $- \frac{4}{\gamma^2}-1 $ is of course compatible with the existence of moments \eqref{existence_moments} and is a well-understood feature of GMC measures which holds for more general covariances and in other dimensions, see section \ref{app3} as well as \cite{review2} for more details. The probability for $Y_{\gamma}$ to be small is given by the term $ \exp ( -( \beta y)^{-\frac{4}{\gamma^2} }) $ which is extremely small and implies that the negative moments of $Y_{\gamma}$ determine its law, a key ingredient of our proof. Progress has been made only very recently in \cite{Mabuchi} to explain the universality behind the probability for a GMC to be small (see also the discussion of \cite{RZhu} in the case of the unit interval).

Before explaining the main ideas behind the proof and the connection with Liouville conformal field theory, we will first enumerate the numerous applications of this result.\\

\noindent

\emph{Acknowledgements}: I would first like to thank R\'emi Rhodes and Vincent Vargas for making me discover Liouville theory and for all their support throughout this project. I would also like to thank Tunan Zhu for many fruitful discussions and for his help on some complicated calculations. Lastly I am grateful to Juhan Aru, Nicolas Curien, Yan Fyodorov, Jon Keating, Gaultier Lambert, Ellen Powell, Avelio Sep\'ulveda, Mo Dick Wong, as well as to the anonymous referees for all their comments that helped me improve this paper.

\subsection{Applications}\label{applications}

\subsubsection{Critical GMC and the maximum of the GFF on the circle}\label{app1}
A problem that has attracted a lot of attention is the behaviour of the maximum of the Gaussian free field on the unit circle. We refer the reader to \cite{Arguin} for a review on the study of extrema of log-correlated processes. The link with GMC theory goes as follows, it is possible to make sense of GMC in the critical case $\gamma=2$ by the so-called derivative martingale construction. In this case, the measure denoted by $-\frac{1}{2} X(e^{i \theta})e^{X(e^{i \theta})}d \theta$ is obtained as the following limit,
\begin{equation}\label{derivconst}
- \frac{1}{2} X(e^{i \theta})e^{X(e^{i \theta})}d \theta := - \underset{\epsilon \to 0} {\lim}   \left(\frac{1}{2} X_\epsilon(e^{i \theta})-  \frac{1}{2} \mathbb{E} [X_\epsilon(e^{i \theta})^2] \right) e^{X_\epsilon(e^{i \theta}) -\frac{1}{2} \mathbb{E}[X_\epsilon(e^{i \theta})^2  ]} d\theta,
\end{equation}
where $X_\epsilon$ is a reasonable cut-off approximation of $X$ which converges to $X$ as $\epsilon$ goes to $0$. The construction \eqref{derivconst} converges to a non trivial random positive measure. This was proved  in \cite{DRSV1,DRSV2} for specific cut-offs $X_\epsilon$ and generalized to general cut-offs in \cite{Powell}. We now introduce: 
\begin{equation}\label{defY'}
 Y'  := - \frac{1}{2} \int_0^{2 \pi} X(e^{i \theta})e^{X(e^{i \theta})} d \theta. 
\end{equation}
It is natural to expect that $Y'$ can be obtained from the sub-critical measures $Y_{\gamma}$ as $\gamma$ goes to $2$ by taking a suitable limit. Indeed it is shown in \cite{APS} that the following holds in probability:
\begin{equation}\label{Othercons}
2 Y'=  \underset{\gamma \to 2}{\lim} \: \frac{1}{2 - \gamma} Y_\gamma.
\end{equation}
From this convergence and Theorem \ref{FyBo}, one can deduce that $2Y'$ has a density $f_{2Y'}$ given by
\begin{equation}\label{}
f_{2Y'}(y)=  y ^{ -2} e^{ -y^{-1 }} \mathbf{1}_{[0, \infty[ }(y).
\end{equation}  
We observe that $\ln 2 Y'$ is distributed like a standard Gumbel law. Recall that an impressive series of works (see \cite{BiLoui,DiRoZei} for the latest results)   have proven that for suitable sequences of cut-off approximations $X_\epsilon$ the following convergence in law  holds
\begin{equation}\label{Convmax}
\max_{ \theta \in [0,2 \pi]}   X_\epsilon(e^{i \theta})- 2 \ln \frac{1}{\epsilon} + \frac{3}{2} \ln \ln \frac{1}{\epsilon}  \underset{\epsilon \to 0}{\rightarrow}  \mathcal{G}+ \ln Y' + C
\end{equation}  
where $\mathcal{G}$ is a Gumbel law independent from $Y'$ and $C$ is a non universal constant that depends on the cut-off procedure. From this convergence and previous considerations, one can deduce the following convergence in law

\begin{equation}\label{Convmaxprecise}
\max_{ \theta \in [0,2 \pi]}   X_\epsilon(e^{i \theta})- 2\ln \frac{1}{\epsilon} + \frac{3}{2} \ln \ln \frac{1}{\epsilon}  \underset{\epsilon \to 0}{\rightarrow}  \mathcal{G}_1 + \mathcal{G}_2 + C
\end{equation}  
where $\mathcal{G}_1$ and $\mathcal{G}_2$ are two independent Gumbel laws and where we have absorbed the factor $\ln 2$ in the constant $C$. This convergence was conjectured in Fyodorov-Bouchaud \cite{FyBo}. As a matter of fact, Fyodorov-Bouchaud state \eqref{Convmaxprecise} as their main result.\footnote{More accurately they expressed the limit density in terms of a modified Bessel function which was noticed by Subag and Zeitouni in \cite{SuZei} to be the sum of two independent Gumbel laws.}  Mathematically, it is the first occurrence of an explicit formula for the limit density of the properly recentered maximum of a GFF.
We nonetheless point out that at the present moment some technical results still need to be rigorously written down in order for \eqref{Convmaxprecise} to be a theorem. In fact, the works \cite{BiLoui,DiRoZei} establish the convergence result \eqref{Convmax} with a variable $Y'$ which has not yet been rigorously proved to be the same as our definition \eqref{defY'} of $Y'$. Nonetheless, private communications with the authors of \cite{BiLoui,DiRoZei} confirm that their methods can be extended to prove that both definitions of $Y'$ coincide.

\subsubsection{Unitary random matrix theory}\label{app2}
A similar story can be told for unitary random matrices. Contrarily to the previous subsection, we emphasize here that novel ideas are still missing in order to be able to use our Theorem \ref{FyBo} to prove the following statements, especially for the analogue of \eqref{Convmaxprecise}. Let $U_N$ denote the $N \times N$ random matrices distributed according to the Haar probability measure on the unitary group $U(N)$. Denoting by $( e^{i \theta_1}, \dots, e^{i \theta_n})$ the eigenvalues of $U_N$, we consider its characteristic polynomial $p_N(\theta)$ evaluated on the unit circle at a point $e^{i\theta}$:
\begin{equation}
 p_N( \theta) = \det ( 1 - e^{- i \theta} U_N) = \prod_{ k=1}^N ( 1 - e^{ i (\theta_k - \theta) }).
\end{equation}
Recently, the following convergence in law has been obtained in \cite{Webb, webb2} for a real $\alpha \in ( -\frac{1}{2}, 2)$:
\begin{equation}\label{analogy}
\frac{\vert p_N(\theta) \vert^{\alpha} }{\mathbb{E}[\vert p_N(\theta) \vert^{\alpha}]} d \theta \underset{N \rightarrow \infty}{\rightarrow} e^{\frac{\vert \alpha \vert}{2} X(e^{i \theta}) } d \theta.
\end{equation} 
This convergence seems to indicate that $2 \ln \vert p_N(\theta) \vert $ should be seen as a cut-off of $X$ just like our $X_{\epsilon}$ with $N$ corresponding to $\frac{1}{\epsilon}$. We thus expect to have for a real $p < \frac{4}{\alpha^2}$:\footnote{This convergence of moments is a consequence of the works \cite{Webb, webb2} in the case where $p \in [0,1]$ or $p \in [0,2]$ when the second moment is finite. The convergence for other values of $p$ seems to require novel estimates.}
\begin{equation}\label{conj_kea}
\mathbb{E}\left[ \left( \frac{1}{2 \pi} \int_0^{2 \pi} \frac{\vert p_N(\theta) \vert^{\alpha} }{\mathbb{E}[\vert p_N(\theta) \vert^{\alpha}]} d \theta\right)^p \right] \underset{N \rightarrow \infty}{\rightarrow} \mathbb{E}\left[ \left(\frac{1}{2 \pi} \int_0^{2 \pi} e^{ \frac{\vert \alpha \vert }{2} X(e^{i \theta})} d \theta \right)^p \right] = \frac{\Gamma(1 - p \frac{\alpha^2}{4})}{\Gamma(1 - \frac{\alpha^2}{4})^p}
\end{equation} 
Notice that $\mathbb{E}[\vert p_N(\theta) \vert^{\alpha} ]$ is independent of $\theta$. Now the following asymptotic is known for $\alpha > -1$,
\begin{equation}\label{conj_kea2}
\mathbb{E}[\vert p_N(\theta) \vert^{\alpha} ] \underset{N \rightarrow \infty}{\sim} \frac{G(1 + \alpha/2)^2}{G(1 + \alpha)} N^{\frac{\alpha^2}{4}}, 
\end{equation}
where $G$ is the so-called Barnes' function. Assuming \eqref{conj_kea} to hold along with \eqref{conj_kea2} establishes the asymptotic conjectured in \cite{FyHiKe} and further studied in \cite{Lambert}, for a real $p < \frac{4}{\alpha^2}$: 
\begin{equation}
\mathbb{E}\left[ \left( \frac{1}{2 \pi} \int_0^{2 \pi} \vert p_N(\theta) \vert^{\alpha}  d \theta\right)^p \right] \underset{N \rightarrow \infty}{\sim} \frac{G(1 + \alpha/2)^{2p}}{G(1 + \alpha)^p} \frac{\Gamma(1 - p \frac{\alpha^2}{4})}{\Gamma(1 - \frac{\alpha^2}{4})^p} N^{ \frac{p \alpha^2}{4}}.
\end{equation}
Now again based on the analogy suggested by \eqref{analogy}, it is reasonable that the properly shifted maximum of $2 \ln \vert p_N( \theta) \vert$ should converge to the same limit as the (properly shifted) maximum of the GFF on the circle. Indeed it has been recently conjectured by Fyodorov, Hiary and Keating \cite{FyHiKe} that the following convergence in law should hold
\begin{equation}\label{Convmaxpn}
\max_{ \theta \in [0,2 \pi]}   \ln | p_N(\theta)  |- \ln N +  \frac{3}{4} \ln \ln N  \underset{N \to \infty}{\rightarrow}  \mathcal{G}_1+ \mathcal{G}_2+C
\end{equation}  
where $\mathcal{G}_1$ and $\mathcal{G}_2$ are again two independent Gumbel laws and $C$ a real constant. On the mathematical side, there has been a series of works \cite{ArBeBour,PaZei,ChMaNaj} aiming at this result. The most recent result \cite{ChMaNaj} establishes that 
\begin{equation}
\max_{ \theta \in [0,2 \pi]}   \ln | p_N(\theta)  |- \ln N + \frac{3}{4} \ln \ln N
\end{equation} 
is tight. Just like for the GFF it is natural to expect that the following convergence is easier to establish directly
\begin{equation}\label{Convmaxpnbis}
\max_{ \theta \in [0,2 \pi]}   \ln | p_N(\theta)  |- \ln N + \frac{3}{4} \ln \ln N  \underset{N \to \infty}{\rightarrow}  \mathcal{G}_1+ \ln Y'+C.
\end{equation}  
Our result could then prove instrumental in precisely identifying the limit in the conjectured  convergence \eqref{Convmaxpn}.

\subsubsection{The tail of GMC in dimension 1}\label{app3}

Finally, Rhodes and Vargas in \cite{tail} introduced a simple method to compute tail expansions for general GMC measures. More precisely, in the 1d case, consider
a log-correlated field $\tilde{X}$ on an open set $\mathcal{O} \subset \partial \mathbb{D}$ with the following covariance,
\begin{equation}
\mathbb{E}[\tilde{X}(e^{i \theta}) \tilde{X}(e^{i \theta'})] = 2 \ln \frac{1}{\vert e^{i \theta} -e^{i \theta'} \vert} +f(e^{i \theta},e^{i \theta'}), 
\end{equation}
for a smooth function $f$. The authors of \cite{tail} showed that the following holds for some $\delta >0$,
\begin{equation}\label{expansion1d}
 \mathbb{P}\left( \int_{\mathcal{O} } e^{ \frac{\gamma}{2} \tilde{X}(e^{i \theta})  } d \theta > t \right) \underset{t \rightarrow \infty}{=} \left( \int_{ \mathcal{O}} e^{ ( \frac{4}{\gamma^2} -1)  f(e^{i \theta},e^{i \theta})} d \theta \right) ( 1 -\frac{\gamma^2}{4})
\frac{R_1(\gamma)}{t^{\frac{4}{\gamma^2}}} +
o(t^{- \frac{4}{\gamma^2} - \delta}),
\end{equation}
where $\mathcal{O}$ is an open subset of $\partial \mathbb{D}$ and $R_1(\gamma)$ is a non explicit universal constant defined in terms of the expectation of a random variable. 
Since Theorem \ref{FyBo} gives an explicit tail expansion for $Y_\gamma$ and the variable $Y_\gamma$ has a tail expansion which satisfies \eqref{expansion1d}, one can deduce an explicit value for $R_1(\gamma)$. This leads to 
\begin{equation}
R_1(\gamma) = \frac{(2 \pi)^{ \frac{4}{\gamma^2} -1  }  }{(1 -
\frac{\gamma^2}{4} ) \Gamma(1 - \frac{\gamma^2}{4})^{ \frac{4}{\gamma^2} }
 }  .
\end{equation}
See also the discussion of \cite{RZhu} on the reflection coefficients for tail expansions of GMC.

\subsection{Strategy of the proof}
To explain the ideas behind our proof of Theorem \ref{FyBo}, we must make a detour in the world of conformal field theory (CFT). In 2014, David, Kupiainen, Rhodes and Vargas \cite{Sphere} applied the theory of GMC to define rigorously Liouville conformal field theory (LCFT) on the Riemann sphere. This theory was first introduced by A. Polyakov in his 1981 seminal paper \cite{Pol} where he proposed a path integral theory of random two dimensional surfaces. In \cite{Sphere} the authors discovered that the correlation functions of LCFT could be expressed as fractional moments of GMC measures with log singularities therefore rendering possible the mathematical study of LCFT. The theory was defined on the Riemann sphere in \cite{Sphere} then on the unit disk in \cite{Disk} and on other surfaces in \cite{Tori}, \cite{Annulus}. Let us also mention the landmark paper \cite{DuSh} by Duplantier-Sheffield establishing the first  probabilistic formulation of the KPZ formula of physics using GMC type measures. This line of work was then pursued further by Duplantier-Miller-Sheffield \cite{Mating} which develops a theory of quantum surfaces equipped with a natural Liouville quantum gravity measure. For a connection with the framework of LCFT presented in section \ref{CFT}, see for instance \cite{Aru} or \cite[Lecture 2]{ReviewDOZZ}.

Since Liouville theory is a CFT, one expects that it is possible to use the framework developed by Belavin, Polyakov and Zamolodchikov (BPZ) in \cite{BPZ} to compute explicitly its correlation functions. As a matter of fact, the original motivation of BPZ for introducing CFT was to compute the correlations of LCFT although it has now grown into a huge field of theoretical physics. Recently, Kupiainen, Rhodes and Vargas were indeed able to rigorously implement the BPZ framework for LCFT in a probabilistic setting. As an output of their constructions, they gave a proof of the celebrated DOZZ formula \cite{DOZZ1,DOZZ2}  for the three-point function of LCFT whose value was conjectured independently by Dorn and Otto in \cite{DO} and by Zamolodchikov and Zamolodchikov in \cite{ZZ}. 

Concerning the strategy of our proof, the key observation is to realize that the inverse moments of GMC integrated on the unit circle can be expressed as one-point correlation functions of LCFT on the unit disk. This link was to the best of our knowledge unknown even to physicists. Thanks to this observation, we can develop the framework of CFT to compute the inverse moment using a strategy similar to the proof of the DOZZ formula \cite{DOZZ1,DOZZ2}. However, working on a domain with boundary requires to introduce a novel BPZ differential equation - see Theorem \ref{BPZequa} below - which differs from the equation of \cite{DOZZ1} on the Riemann sphere.

Let us now introduce some notations which will be used in the sequel. Let $p$ be a real number such that $p< \frac{4}{\gamma^2}$. We denote:
\begin{equation}
U( \gamma, p ) = \mathbb{E}\left[ \left(\int_0^{2 \pi} e^{ \frac{\gamma}{2} X(e^{i \theta})} d \theta\right)^p \right].
\end{equation}
The proof is based on studying the following function or ``observable" defined for $t \in [0,1]$:
\begin{equation}\label{defGp}
G(\gamma, p, t) = \mathbb{E}\left[ \left(\int_0^{2 \pi} \vert t - e^ {i \theta} \vert^{\frac{\gamma^2}{2}  } e^{ \frac{\gamma}{2} X(e^{i \theta})} d \theta \right)^p \right].
\end{equation} 
At first glance, it can seem mysterious why the introduction of $G(\gamma, p, t)$ can be of any help in computing $U( \gamma, p )$. In order to understand why $G(\gamma, p, t)$ is the ``right" auxiliary function to look at, one must cast the problem in the language of LCFT with boundary. It turns out that the moment $U( \gamma, p )$ is the one-point correlation function and that the function $G(\gamma, p, t)$ is the two-point correlation function with a so-called degenerate field insertion, see section \ref{CFT} for the definitions.  Therefore the function $G(\gamma, p, t)$ is expected to obey a differential equation known as the BPZ equation. Indeed, we will prove in section \ref{CFT} using probabilistic techniques that:

\begin{proposition}\label{EquaG}
(BPZ) For $\gamma \in (0,2)$ and $p<0$ the function $ t \mapsto G(\gamma,p, t) $ satisfies the following differential equation:
$$ \left( t (1 - t^2) \frac{\partial^2}{\partial t^2} +(t^2-1)\frac{\partial}{\partial t} + 2( C - (A +B+1)t^2) \frac{\partial}{\partial t} - 4ABt \right) G(\gamma,p, t) = 0    $$
with the following values for $A$, $B$, and $C$:
$$ A = -\frac{\gamma^2 p}{4}, \text{  } B = - \frac{\gamma^2}{4}, \text{  }  C = \frac{\gamma^2}{4}(1 - p) + 1.  $$
\end{proposition}
Here the hypothesis on $p$ is purely technical and could be relaxed with little effort. A simple change of variable $x = t^2$ and $G(\gamma,p, t) = H(x)$ turns the BPZ equation for $G(\gamma,p, t)$ into a hypergeometric equation for $H(x)$
$$ \left( x(1-x) \frac{\partial^2}{\partial x^2} + (C -(A +B +1)x) \frac{\partial}{\partial x} - AB \right) H(x) = 0. $$
The solution space of this equation is two dimensional. In fact, we can give two sets of solutions, one corresponding to an expansion in powers of $x$ and the other to an expansion in powers of $1 - x$. From the theory of hypergeometric functions recalled in appendix \ref{apphyp} we thus obtain:

\begin{proposition}\label{CB}
For $\gamma \in (0,2)$ with $\gamma \neq \sqrt{2}$\footnote{The value $\gamma = \sqrt{2}$ is excluded here for technical reasons corresponding to poles in the Gamma functions of $F$ and of \eqref{hpy1}. This value will be recovered in the proof of Proposition \ref{Shift} by a standard continuity argument in $\gamma$.} and $p<0$, we have
$$ G(\gamma, p , t) = C_1 F( -\frac{\gamma^2 p}{4} , - \frac{\gamma^2}{4}, \frac{\gamma^2}{4}(1 -p) +1,t^2 ) + C_2 t^{\frac{\gamma^2}{2} (p-1) } F( - \frac{\gamma^2}{4}, \frac{\gamma^2}{4} (p-2) , \frac{\gamma^2}{4}(p-1) + 1,t^2 ) $$
and
$$ G(\gamma, p , t) = B_1 F( - \frac{\gamma^2 p}{4} , - \frac{\gamma^2 }{4}, - \frac{\gamma^2}{2},1 -t^2) + B_2 (1-t^2)^{ 1 + \frac{\gamma^2}{2}    } F ( 1 +  \frac{\gamma^2}{4}, \frac{\gamma^2}{4} (2-p) +1 , 2 + \frac{\gamma^2}{2},1 -t^2 ) $$
where $F$ is the standard hypergeometric series. The coefficients $C_1$, $C_2$, $B_1$ and $B_2$ are real constants that depend on $\gamma$ and $p$. Since the solution space of the hypergeometric equation is two dimensional, the coefficients  are linked by  the explicit change of basis formula \eqref{hpy1} written in appendix \ref{apphyp}.
\end{proposition}

The end of the proof performed in section \ref{sec_shift} is based on exploiting the fact that the above coefficients $C_1$, $C_2$, $B_1$ and $B_2$  can be identified in terms of $U( \gamma, p )$ by performing asymptotic expansions directly on the expression \eqref{defGp} of $G(\gamma, p, t)$. Notice for instance that $ C_1 = G(\gamma,p,0) = U(\gamma,p)$. We also express $B_2$ in terms of $U( \gamma, p-1)$ and find $C_2 =0$. Using the change of basis formula \eqref{hpy1} we thus show at the end of section \ref{sec_shift} the following shift equation for $U( \gamma, p )$: 

\begin{proposition}\label{Shift}
For all $ \gamma \in (0,2)$ and for $p \leq 0$, we have the relation
\begin{equation}\label{Shift_U}
U(\gamma, p ) =  \frac{2 \pi \Gamma(1 - p \frac{\gamma^2}{4} ) }{ \Gamma(1 - \frac{\gamma^2}{4}) \Gamma(1 - (p-1) \frac{\gamma^2}{4} )} U(\gamma, p-1 ).
\end{equation} 
\end{proposition}

From this shift equation we deduce recursively all the positive moments of the variable $\frac{1}{Y_\gamma}$, i.e. we get 
\begin{equation}
 \mathbb{E}[Y_\gamma^{-n} ] = \Gamma(1 + \frac{n \gamma^2}{4}  ) \Gamma( 1 - \frac{\gamma^2}{4} )^n, \quad \forall n \in \mathbb{N}.
\end{equation}
The series
$$ \lambda \mapsto \sum_{n=0}^{\infty} \frac{\lambda^n}{n !} \Gamma(1 + \frac{n \gamma^2}{4}  ) \Gamma( 1 - \frac{\gamma^2}{4} )^n $$
has an infinite radius of convergence, meaning that the moments of $\frac{1}{Y_\gamma}$ entirely determine its law and one can even give an explicit probability density for $\frac{1}{Y_{\gamma}}$,
\begin{equation}
f_{\frac{1}{Y_\gamma}}(y) =  \frac{4}{\beta \gamma^2} ( \frac{y}{\beta})^{ \frac{4}{\gamma^2}-1} e^{ -( \frac{y}{\beta})^{\frac{4}{\gamma^2} }} \mathbf{1}_{[0, \infty[ }(y),  \quad 
\end{equation}
where $\beta = \Gamma( 1 - \frac{\gamma^2}{4})$. It can easily be turned into a probability density for $Y_\gamma$,
\begin{equation}
f_{Y_\gamma}(y) =  \frac{4 \beta }{\gamma^2} ( \beta y )^{ - \frac{4}{\gamma^2}-1} e^{ -( \beta y)^{-\frac{4}{\gamma^2} }} \mathbf{1}_{[0, \infty[ }(y),
\end{equation}
which proves Theorem \ref{FyBo}. Also notice that our proof does not use the value of the Morris integral \eqref{keyobs} and in fact we give a new proof of its value by taking integer moments in our GMC measure. 

Lastly, we point out that we have actually done more than just compute $U(\gamma,p)$, we have also completely determined the function $G(\gamma,p,t)$. By choosing $t=1$ we obtain the following corollary:

\begin{corollary}\label{corol}
Let $\gamma \in (0,2)$. For all real $p$ such that $p < \frac{4}{\gamma^2}$ the following identity holds: 
\begin{equation} 
\mathbb{E}\left[\left(\frac{1}{2 \pi} \int_0^{2 \pi} \vert 1 - e^{i \theta}\vert^{\frac{\gamma^2}{2}} e^{\frac{\gamma}{2}X(e^{i \theta})} d\theta  \right)^p \right] =  \frac{\Gamma(1 - p \frac{\gamma^2}{4}) \Gamma( 1 + \frac{\gamma^2}{2}) \Gamma(1 + (1-p)\frac{\gamma^2}{4} )  }{\Gamma(1 - \frac{\gamma^2}{4})^p \Gamma( 1 + \frac{\gamma^2}{4}) \Gamma(1 + (2-p)\frac{\gamma^2}{4} )   }.
\end{equation}
Equivalently we also have
\begin{equation}
\frac{1}{2 \pi} \int_0^{2 \pi} \vert 1 - e^{i \theta}\vert^{\frac{\gamma^2}{2}} e^{\frac{\gamma}{2}X(e^{i \theta})} d\theta \overset{law}{=} Y_{\gamma} X_1^{- \frac{\gamma^2}{4}}
\end{equation}
with $Y_{\gamma}$, $X_1$ independent and $X_1 \sim \mathcal{B}(1 + \frac{\gamma^2}{4},  \frac{\gamma^2}{4} )$, where $\mathcal{B}(\alpha,\beta)$ denotes the standard beta law.
\end{corollary}
A similar formula is also expected to hold for the so-called dual degenerate insertion,
\begin{conjecture}\label{conj} Let $\gamma \in (0,2)$. For all real $p$ such that $p < \frac{4}{\gamma^2}$ we expect to have 
\begin{equation}
\mathbb{E}\left[\left(\frac{1}{2 \pi} \int_0^{2 \pi} \vert 1 - e^{i \theta}\vert^{2} e^{\frac{\gamma}{2}X(e^{i \theta})} d\theta  \right)^p\right] =  \frac{\Gamma(1 - p \frac{\gamma^2}{4}) \Gamma( 1 + \frac{8}{\gamma^2}) \Gamma(1 + \frac{4}{\gamma^2} -p )  }{\Gamma(1 - \frac{\gamma^2}{4})^p \Gamma( 1 + \frac{4}{\gamma^2}) \Gamma(1 + \frac{8}{\gamma^2} -p )   }
\end{equation}
and we can write again
\begin{equation}
\frac{1}{2 \pi} \int_0^{2 \pi} \vert 1 - e^{i \theta}\vert^{2} e^{\frac{\gamma}{2}X(e^{i \theta})} d\theta \overset{law}{=} Y_{\gamma} X_2^{- 1}
\end{equation}
with $Y_{\gamma}$, $X_2$ independent and $X_2 \sim \mathcal{B}(1 + \frac{4}{\gamma^2}, \frac{4}{\gamma^2} )$.
\end{conjecture}
In fact we expect that using similar techniques it will be possible to obtain many more exact formulas on GMC measures. A analogue result has very recently been obtained for the total mass of the GMC measure on the unit interval $[0,1]$ with arbitrary insertions in $0$ and $1$, see \cite{RZhu}. One could also study more general cases on the unit circle, such as Conjecture \ref{conj} which is a special case of the more general conjectures of \cite{Ostro2}. Therefore our methods combined with the proof of the DOZZ formula \cite{DOZZ1, DOZZ2} open up brand new perspectives for studying the integrability of GMC measures. Lastly we point out that very recently an alternative proof to the Fyodorov-Bouchaud formula was obtained in \cite{ChNaj} by showing an equality between GMC on the unit circle and the limit of the circular beta ensemble. This alternative approach relies on the theory of orthogonal polynomials and is also able to describe the Fourier coefficients of the GMC measure on the circle.

\section{Boundary Liouville Conformal Field Theory}\label{CFT}

\subsection{The Liouville correlation functions on $\mathbb{H}$}
In order to prove Proposition \ref{EquaG} we introduce the framework of LCFT on a domain with boundary, following the setting of \cite{Disk}. Here we will work on the upper half plane $\mathbb{H}$ (with boundary $\partial \mathbb{H} = \mathbb{R}$) which simplifies computations since the terms obtained will behave nicely under conjugation $z \mapsto \overline{z}$. We can then transpose everything easily to the unit disk $\mathbb{D}$ by the KPZ relation \eqref{KPZ}. The starting point is the well known Liouville action where in our case we must add a boundary term,
\begin{equation}\label{action}
S_L(X,\hat{g}) = \frac{1}{4\pi} \int_{\mathbb{H}} ( |\partial^{\hat{g}} X|^2 + Q R_{\hat{g}} X ) \hat{g}(z) d^2z  + \frac{1}{2\pi} \int_{  \mathbb{R}} (  Q K_{\hat{g}} X + 2 \pi \mu_{\partial} \; e^{\frac{\gamma}{2} X}) \hat{g}(s)^{1/2} ds,\footnote{The action usually also contains a bulk interaction term $\mu e^{\gamma X} $ but for our purposes we set $\mu =0$. Hence we are working with a degenerate form of boundary LCFT.} 
\end{equation}
where $\partial^{\hat{g}}$, $R_{\hat{g}}$, and $K_{\hat{g}}$  respectively stand for the gradient, Ricci scalar curvature and geodesic curvature of the boundary in the metric $ \hat{g}$ (which can be chosen arbitrarily). We also have $\gamma \in (0,2)$, $ Q = \frac{\gamma}{2} + \frac{2}{\gamma} $, and $ \mu_{\partial} > 0$. In the following we choose the background metric $\hat{g}(z) = \frac{4}{\vert z + i \vert^4}$ on $\mathbb{H}$. This is a natural choice as if we map $\mathbb{H}$ to $\mathbb{D}$ (using $z \mapsto \frac{z-i}{z+i}$) this choice corresponds to the Euclidean (or flat) metric on $\mathbb{D}$. With our choice of $\hat{g}$ we get $R_{\hat{g}} = 0 $ and $ K_{\hat{g}} = 1$ which simplifies the expression of \eqref{action}.

With this action we can formally define the correlation functions of LCFT. They are the quantities of interest of the theory that we hope to be able to compute with the techniques of CFT. We will consider two types of insertion points in our correlations: bulk insertions $(z_i, \alpha_i)$ (with $z_i \in \mathbb{H}$ and $\alpha_i \in \mathbb{R}$) and boundary insertions $(s_j, \beta_j)$ (with $s_j \in \mathbb{R}$ and  $\beta_j \in \mathbb{R}$). We introduce the following notations for the so-called vertex operators:
\begin{align*}
& V_{\alpha_i}(z_i) =  e^{ \alpha_i ( X(z_i) + \frac{Q}{2} \ln \hat{g}(z_i) ) } \\
& V_{\beta_j}(s_j) =  e^{ \frac{ \beta_j}{2} ( X(s_j) + \frac{Q}{2} \ln \hat{g}(s_j)  ) }.
\end{align*}
We formally define the correlations by,
\begin{equation}\label{def_formal}
\langle \prod_{i=1}^N V_{\alpha_i}(z_i) \prod_{j=1}^M V_{\beta_j}(s_j) \rangle_{\mathbb{H}}  =  \int_{\Sigma} D_{\hat{g}}X \prod_{i=1}^N e^{ \alpha_i ( X(z_i) + \frac{Q}{2} \ln \hat{g}(z_i) ) }  \prod_{j=1}^M e^{ \frac{\beta_j}{2} ( X(s_j) + \frac{Q}{2} \ln \hat{g}(s_j)  ) }  e^{-S_L(X,\hat{g})},
\end{equation}
for $N,M$ in $\mathbb{N}$. The philosophy of this heuristic definition is the following. Starting from a formal uniform measure $D_{\hat{g}}X$ on the space of maps $\Sigma = \lbrace X : \mathbb{H} \mapsto \mathbb{R} \rbrace $, we add a density given by $e^{-S_L(X,\hat{g})}$. This is simply the Boltzmann weight framework of statistical physics where the probability of a given state (here a map X) is proportional to exponential minus its energy (here the Liouville action). Following \cite{Sphere, Disk} it turns out that it is possible to give a rigorous probabilistic definition to \eqref{def_formal} in terms of GMC measures. The idea is to interpret the quantity $e^{-\frac{1}{4 \pi} \int_{\mathbb{H}}  |\partial^{\hat{g}} X|^2  \hat{g}(z) dz^2   } D_{\hat{g}}X$ as the formal density of a GFF that will be denoted $X$. As explained in \cite{ Disk}, the correct choice of boundary conditions for $X$ that lead to a conformally invariant functional are the Neumann boundary conditions. Furthermore, since $X$ is a priori only defined  up to a global constant we fix this constant by requiring the average of $X$ with respect to $\hat{g}$ over the boundary of $\mathbb{H}$ to be $0$. Thus one obtains the centered Gaussian field $X$ on $\mathbb{H}$ with covariance given for $x, y \in \mathbb{H}$ by,\footnote{The covariance of the GFF $X$ on $\mathbb{H}$ differs from the covariance \eqref{Cov} of $X$ on $\mathbb{D}$. The GFF on $\mathbb{D}$ is the image of the GFF on $\mathbb{H}$ by the conformal map $z \mapsto \frac{z-i}{z+i}$ linking $\mathbb{H}$ and $\mathbb{D}$.}
\begin{equation}\label{cov_h}
\mathbb{E} [ X(x) X(y) ] = \ln \frac{1}{\vert x - y \vert \vert x - \overline{y} \vert } + \ln \vert x + i \vert^2 +  \ln \vert y + i \vert^2 - 2 \ln 2,
\end{equation}
and with our choice of normalization for $X$ (using $ K_ { \hat{g} } =1$):
\begin{equation}\label{av}
 \int_{\mathbb{R}} X(s) \hat{g}(s)^{1/2} ds = \int_{\mathbb{R}} K_ { \hat{g} }(s) X(s) \hat{g}(s)^{1/2} ds =0.
\end{equation}
Since $X$ lives in the space of distributions we will need again to introduce a cut-off or regularization procedure. For $\delta >0$ let:
$$ \mathbb{H}_{\delta} = \lbrace z \in \mathbb{H} \vert \Im(z) > \delta \rbrace.  $$
Then let $\rho : [0, + \infty)  \mapsto [0, + \infty)$ be a $\mathcal{C}^{ \infty}$ function with compact support in $[0,1]$ and such that $ \pi \int_0^{ \infty } \rho(t) dt =1  $. For $ x \in \mathbb{H}$ we write  $ \rho_{ \epsilon }(x) = \frac{1}{\epsilon^2} \rho( \frac{x \overline{x}}{\epsilon^2} ) $. Then for $z \in \mathbb{H}_{\delta}$ and $\epsilon  < \delta$, we define $X_{\epsilon}$ by
\begin{equation}\label{reg1}
X_{\epsilon}(z) = ( \rho_{ \epsilon } \ast X)(z) = \int_{\mathbb{H}} d^2 x X(x) \rho_{ \epsilon }(z-x),
\end{equation}
and for $s \in \mathbb{R} $ by
\begin{equation}\label{reg2}
 X_{ \epsilon }(s) = 2( \rho_{\epsilon} \ast X)(s) =2 \int_{\mathbb{H}} d^2 x X(x) \rho_{\epsilon}( s-x).
\end{equation}
The idea of our regularization is that for a point $z \in \mathbb{H}_{\delta}$ at a distance at least $\delta$ from the boundary and for $\epsilon < \delta$ we can smooth our field $X(z)$ with $\rho$ on a ball of radius $\epsilon$ around $z$. For a point $s \in \mathbb{R}$ we will always write our convolution on the half ball contained in $\mathbb{H}$ of size $\epsilon$. Following the results of \cite{Disk}, we now define the correlation functions by the following limit,
\begin{equation}\label{limit23}
\langle \prod_{i=1}^N V_{\alpha_i}(z_i) \prod_{j=1}^M V_{\beta_j}(s_j) \rangle_{\mathbb{H}}   = \lim_{\epsilon \rightarrow 0} \: \langle \prod_{i=1}^N V_{\alpha_i}(z_i)  \prod_{j=1}^M V_{\beta_j}(s_j) \rangle_{\mathbb{H}, \epsilon},
\end{equation}
where
\begin{align}\label{reg_partition}
 \langle \prod_{i=1}^N V_{\alpha_i}(z_i) \prod_{j=1}^M V_{\beta_j}(s_j) \rangle_{\mathbb{H}, \epsilon}  = & \int_{\mathbb{R}}  dc e^{-Qc}  \mathbb{E} \bigg[\prod_{i=1}^N \epsilon^{\frac{\alpha_i^2}{2}} e^{ \alpha_i (X_{\epsilon}(z_i) + \frac{Q}{2} \ln \hat{g}(z_i) +c)} \prod_{j=1}^M \epsilon^{\frac{\beta_j^2}{4}}  e^{ \frac{\beta_j}{2}   (X_{\epsilon}(s_j) + \frac{Q}{2} \ln \hat{g}(s_j) +c)} \nonumber  \\ 
& \times \exp\left(-\mu_{\partial} \int_{  \mathbb{R}} \epsilon^{\frac{\gamma^2}{4}} e^{\frac{\gamma}{2} (X_{\epsilon}(s) + \frac{Q}{2} \ln \hat{g}(s) +c  )}  ds \right) \bigg]  .
\end{align}
At first glance the above definition may appear to be convoluted but it is simply the consequence of replacing the $e^{-\frac{1}{4 \pi} \int_{\mathbb{H}}  |\partial^{\hat{g}} X|^2  \hat{g}(z) dz^2   } D_{\hat{g}}X$ in \eqref{def_formal} by the expectation of the GFF, i.e. $X$ now becomes $X +c$, where $X$ is our GFF with covariance \eqref{cov_h} and $c$ is a constant integrated with respect to the Lebesgue measure on $\mathbb{R}$. As performed in \cite{Disk}, notice that for every exponential of $X_{\epsilon}$ we accompany it with a suitable power of $\epsilon$ to compensate the divergence of $\mathbb{E}[X_{\epsilon}^2]$ in order to have a well-defined limit as $\epsilon \rightarrow 0$. The $c$ is called the zero mode in physics, it corresponds to the fact that $ |\partial^{\hat{g}} X|^2$ only determines the field up to a constant. It is required to construct a conformally invariant functional that does not depend on the choice we have made for $X$ through \eqref{av}, see \cite{Sphere} for more details. To obtain \eqref{reg_partition} we have also used \eqref{av} and the explicit values of $R_{\hat{g}}$ and $K_{\hat{g}}$. Following the results of \cite{Disk}, the limit \eqref{reg_partition}  exists and is non zero if and only if the insertions obey the Seiberg bounds which are:\footnote{Note that in our case we are only working with the boundary GMC measure so the extra condition $\alpha_i <Q$ present in \cite{Disk} does not appear here since the GMC does not have to integrate the insertions in the bulk.}
\begin{equation}\label{Seiberg}
\sum_{i=1}^N \alpha_i + \frac{1}{2} \sum_{j=1}^M \beta_j > Q \text{  } \text{  and  } \text{  } \forall j, \quad \beta_j < Q.
\end{equation}

\subsection{The BPZ differential equation}

Our goal is to prove the following result:

\begin{theorem}\label{BPZequa}
Let $\gamma \in (0,2)$ and $\alpha  > Q + \frac{\gamma}{2}$. Then $ (z_1, z) \mapsto \langle V_{-\frac{\gamma}{2}}(z)  V_{\alpha}(z_1)   \rangle_{\mathbb{H}} $ is $\mathcal{C}^2$ on the set $\lbrace z_1,z \in \mathbb{H} \vert z_1 \neq z \rbrace$ and is a solution of the following PDE
$$ \left( \frac{4}{\gamma^2} \partial_{zz} +  \frac{\Delta_{-\frac{\gamma}{2}}}{(z - \overline{z})^2} +   \frac{\Delta_{\alpha}}{(z - z_1)^2} +  \frac{\Delta_{\alpha}}{(z - {\overline{z_1})^2}}  +  \frac{1}{z - \overline{z}} \partial_{\overline{z}} +  \frac{1}{z - z_1} \partial_{z_1} + \frac{1}{z - \overline{z_1}} \partial_{\overline{z_1}}  \right) \langle V_{-\frac{\gamma}{2}}(z)  V_{\alpha}(z_1)   \rangle_{\mathbb{H}} = 0$$
where $Q = \frac{\gamma}{2} + \frac{2}{\gamma} $,  $\Delta_{\alpha} = \frac{\alpha}{2}( Q - \frac{\alpha}{2})$ and $\Delta_{-\frac{\gamma}{2}} = -\frac{\gamma}{4}( Q + \frac{\gamma}{4})$.
\end{theorem}

First let us note that the condition $\alpha  > Q + \frac{\gamma}{2}$ corresponds exactly to the Seiberg bounds \eqref{Seiberg} and that the $\Delta_{\alpha}, \Delta_{-\frac{\gamma}{2}}$ are the so-called conformal weights of CFT. Before proceeding with the proof of Theorem \ref{BPZequa}, we must derive several identities. First of all, we need to replace all the $\mathbb{E}[X(x) X(y)]$ by the exact log kernel $\ln \frac{1}{\vert x - y \vert \vert x - \overline{y} \vert} $ or in other words we need to eliminate the dependence on the background metric $\hat{g}$. This will be a consequence of the following identity:
\begin{lemma}\label{lem1} For insertions $(z_i, \alpha_i)$ and $(s_j, \beta_j)$ satisfying the Seiberg bounds \eqref{Seiberg} we have:
$$   \frac{\mu_{\partial} \gamma  }{2 } \int_{ \mathbb{R} } ds \langle V_{\gamma } (s) \prod_{i=1}^N V_{ \alpha_i } (z_i) \prod_{j=1}^M V_{\beta_j } (s_j)    \rangle_{\mathbb{H}}  = \left(\sum_{i=1}^N \alpha_i + \sum_{j=1}^M \frac{\beta_j}{2} - Q \right)\langle  \prod_{i=1}^N V_{ \alpha_i } (z_i) \prod_{j=1}^M V_{\beta_j } (s_j)    \rangle_{\mathbb{H}}. $$
\end{lemma}
\begin{proof}
We perform the change of variable $ \frac{2}{\gamma} \ln \mu_{\partial} + c = c' $ in the following expression:
\begin{align*}
  & \int_{\mathbb{R}}  dc e^{-Qc}  \mathbb{E} \bigg[\prod_{i=1}^N \epsilon^{\frac{\alpha_i^2}{2}} e^{ \alpha_i (X_{\epsilon}(z_i) + \frac{Q}{2} \ln \hat{g}(z_i) +c)} \prod_{j=1}^M \epsilon^{\frac{\beta_j^2}{4}}  e^{ \frac{\beta_j}{2}   (X_{\epsilon}(s_j) + \frac{Q}{2} \ln \hat{g}(s_j) +c)} \nonumber  \\ 
& \quad \quad \times \exp\left(-\mu_{\partial} \int_{  \mathbb{R}} \epsilon^{\frac{\gamma^2}{4}} e^{\frac{\gamma}{2} (X_{\epsilon}(s) + \frac{Q}{2} \ln \hat{g}(s) +c )}  ds \right) \bigg]\\
& = \mu_{ \partial }^{ - \frac{2 \sum_i \alpha_i + \sum_j \beta_j -2 Q}{\gamma}}  \int_{\mathbb{R}}  dc' e^{-Qc'}  \mathbb{E} \bigg[\prod_{i=1}^N \epsilon^{\frac{\alpha_i^2}{2}} e^{ \alpha_i (X_{\epsilon}(z_i) + \frac{Q}{2} \ln \hat{g}(z_i) +c')} \prod_{j=1}^M \epsilon^{\frac{\beta_j^2}{4}}  e^{ \frac{\beta_j}{2}   (X_{\epsilon}(s_j) + \frac{Q}{2} \ln \hat{g}(s_j) +c')} \nonumber  \\ 
& \quad \quad \times \exp\left(- \int_{  \mathbb{R}} \epsilon^{\frac{\gamma^2}{4}} e^{\frac{\gamma}{2} (X_{\epsilon}(s) + \frac{Q}{2} \ln \hat{g}(s) +c'  )}  ds \right) \bigg].
\end{align*}

The result is then obtained by differentiating with respect to $\mu_{\partial}$ and then taking $\epsilon$ to $0$.
\end{proof}
So far we have introduced correlation functions of Liouville theory with an arbitrary number of insertions points. For the purposes of Theorem \ref{BPZequa} we will only need to consider correlations with two bulk insertions $z,z_1 \in\mathbb{H}$ of weights $- \frac{\gamma}{2}$ and $ \alpha$ and eventually boundary insertions $s,t \in \mathbb{R}$ of weight $\gamma$. The value $- \frac{\gamma}{2}$ is called the degenerate weight in the language of CFT. It is for this specific value (and also for the dual weight $-\frac{2}{\gamma}$) that a correlation function containing $V_{-\frac{\gamma}{2}}(z)$ will obey a BPZ differential equation. In the forthcoming computations we will extensively use the shorthand notations:

\vspace{8pt}
\begin{tabular}{ll}
$\langle z,z_1 \rangle := \langle V_{\alpha}(z_1) V_{ -\frac{\gamma}{2}  }(z)  \rangle_{\mathbb{H}}$, & $\langle z,z_1 \rangle_{\epsilon} := \langle V_{\alpha}(z_1) V_{ -\frac{\gamma}{2}  }(z)  \rangle_{\mathbb{H}, {\epsilon}}$, \\ \vspace{2pt}
 $\langle s, z, z_1 \rangle := \langle V_{\gamma } (s)  V_{\alpha}(z_1) V_{ -\frac{\gamma}{2}  }(z)  \rangle_{\mathbb{H}}$,& $\langle s, z, z_1 \rangle_{\epsilon} := \langle V_{\gamma } (s)  V_{\alpha}(z_1) V_{ -\frac{\gamma}{2}  }(z)    \rangle_{\mathbb{H},\epsilon}$, \\ \vspace{2pt}
 $\langle s, t, z, z_1 \rangle := \langle V_{\gamma } (s) V_{\gamma } (t)   V_{ \alpha } (z_1) V_{ - \frac{\gamma}{2}} (z)   \rangle_{\mathbb{H}},$&$\langle s, t, z, z_1 \rangle_{\epsilon} := \langle V_{\gamma } (s) V_{\gamma } (t)   V_{ \alpha } (z_1) V_{ - \frac{\gamma}{2}} (z)   \rangle_{\mathbb{H}, \epsilon}.  $
\end{tabular}
\vspace{8pt}

We now derive a relation obtained by integration by parts on the underlying Gaussian measure of our field $X$. We introduce $X(f) = \int_{\mathbb{H}} X(x) f(x) d^2x$, the distributional pairing of $X$ and $f: \mathbb{H} \mapsto \mathbb{R}$ for some smooth $f$ with compact support, as well as the following quantity:
\begin{align}\label{Guassian_IPP}
\langle X(f)  V_{\alpha}(z_1) V_{ -\frac{\gamma}{2}  }(z)    \rangle_{\mathbb{H}, \epsilon} := & \int_{\mathbb{R}}  dc e^{-Qc}  \mathbb{E} \bigg[ X(f) \epsilon^{\frac{\alpha^2}{2}} e^{ \alpha (X_{\epsilon}(z_1) + \frac{Q}{2} \ln \hat{g}(z_1) +c)}  \epsilon^{\frac{\gamma^2}{8}}  e^{ -\frac{\gamma}{2}   (X_{\epsilon}(z) + \frac{Q}{2} \ln \hat{g}(z) +c)} \nonumber \nonumber \\ 
& \times \exp\left(-\mu_{\partial} \int_{  \mathbb{R}} \epsilon^{\frac{\gamma^2}{4}} e^{\frac{\gamma}{2} (X_{\epsilon}(s) + \frac{Q}{2} \ln \hat{g}(s) + c )}  ds \right) \bigg].
\end{align}
A convenient way to see the Gaussian integration by parts is to fix a small parameter $u>0$ and to write the following Girsanov theorem stated in \cite[Theorem 6.1]{ReviewDOZZ} for the Gaussian random variable $u X(f)$,
$$ \mathbb{E} \left[ e^{u X(f) - \frac{u^2}{2} \mathbb{E}[X(f)^2]} F( (X(x))_{x \in \mathbb{H}} ) \right] = \mathbb{E}\left[ F((X(x) + u \mathbb{E}[X(x) X(f)])_{x \in \mathbb{H}}) \right], $$
where $F$ is a continuous bounded functional of the GFF. Then by differentiating on both sides with respect to $u$ and taking $u =0$:
\begin{equation}\label{IPP_girsanov}
\mathbb{E} \left[ X(f)  F( (X(x))_{x \in \mathbb{H}} ) \right] = {\frac{d}{d u}}_{|u=0} \mathbb{E}\left[ F((X(x) + u \mathbb{E}[X(x) X(f)])_{x \in \mathbb{H}}) \right].
\end{equation}
We can then apply \eqref{IPP_girsanov} to \eqref{Guassian_IPP} by taking $F$ to be the functional of $X$ defining \eqref{Guassian_IPP}. The derivation with respect to $u$ will give three terms, one contribution from $e^{\alpha X_{\epsilon}(z_1)}$, one from $e^{ -\frac{\gamma}{2}   X_{\epsilon}(z) }$, and one from the GMC in the exponential. This last contribution gives a term in $ \langle s, z,z_1 \rangle_{\epsilon} $ as differentiating the exponential of \eqref{Guassian_IPP} containing the GMC has the effect of producing an insertion $s$ on the boundary of weight $\gamma$. Thus the relation:
\begin{align}\label{Gaussian_IPP}
&\langle X(f)  V_{\alpha}(z_1) V_{ -\frac{\gamma}{2}  }(z)    \rangle_{\mathbb{H}, \epsilon} \nonumber \\
 & =  \alpha \mathbb{E}[X(f) X_{\epsilon}(z_1)] \langle z, z_1 \rangle_{ \epsilon } - \frac{\gamma}{2} \mathbb{E}[X(f) X_{\epsilon}(z)] \langle z, z_1 \rangle_{ \epsilon } - \mu_{\partial} \frac{\gamma}{2} \int_{ \mathbb{R}} \mathbb{E}[X(f) X_{\epsilon}(s)] \langle s, z, z_1 \rangle_{ \epsilon } ds.
\end{align}
Our goal is now to compute the derivatives of $ \langle V_{\alpha}(z_1) V_{ -\frac{\gamma}{2}  }(z) \rangle_{\mathbb{H} } $ in order to prove Theorem \ref{BPZequa}. We will illustrate how this computation works with $ \partial_z \langle V_{\alpha}(z_1) V_{ -\frac{\gamma}{2}  }(z) \rangle_{\mathbb{H} } $. We must use our regularization procedure so we fix $\delta >0$ and choose $z,z_1 \in \mathbb{H}_{\delta}$ and $\epsilon< \delta$. We compute the derivative of the regularized partition function \eqref{reg_partition} with respect to $z$:
\begin{align*}
 \partial_z \langle z,z_1 \rangle_{\epsilon}  &=  -\frac{\gamma}{2} \int_{\mathbb{R}}  dc e^{-Qc}  \mathbb{E} \bigg[  \partial_z ( X_{\epsilon}(z) + \frac{Q}{2} \ln \hat{g}(z)) \epsilon^{\frac{\alpha^2}{2}} e^{ \alpha (X_{\epsilon}(z_1) + \frac{Q}{2} \ln \hat{g}(z_1) +c)} \epsilon^{\frac{\gamma^2}{8}} e^{ -\frac{\gamma}{2} (X_{\epsilon}(z) + \frac{Q}{2} \ln \hat{g}(z) +c)} \\
 & \: \: \times \exp\left(-\mu_{\partial} \int_{  \mathbb{R}} \epsilon^{\frac{\gamma^2}{4}} e^{\frac{\gamma}{2} (X_{\epsilon}(s)+c + \frac{Q}{2} \ln \hat{g}(s)  )} ds \right) \bigg] \\
& = -\frac{\gamma}{2} \langle (\partial_z X_{\epsilon}(z)) V_{\alpha}(z_1) V_{ -\frac{\gamma}{2}  }(z) \rangle_{\mathbb{H}, \epsilon} - \frac{\gamma Q}{4} (\partial_z \ln \hat{g}(z)) \langle z, z_1 \rangle_{\mathbb{H}, \epsilon}
\end{align*}

Since $\partial_z X_{\epsilon}(z) = \int_{\mathbb{H}} d^2x X(x) \partial_z \rho_{\epsilon}(z-x)   $, the term $ \langle (\partial_z X_{\epsilon}(z)) V_{\alpha}(z_1) V_{ -\frac{\gamma}{2}  }(z) \rangle_{\mathbb{H}, \epsilon}$ should be understood by the expression \eqref{Guassian_IPP} with  $f(x) = \partial_z \rho_{\epsilon}(z-x)$ and can thus be evaluated with \eqref{Gaussian_IPP}. Using \eqref{cov_h}, \eqref{reg1}, \eqref{reg2}, we compute:
\begin{align*}
 \mathbb{E}[\partial_z X_{\epsilon}(z) X_{\epsilon}(z_1)] &= \int_{\mathbb{H}} \int_{\mathbb{H}} d^2x d^2y \mathbb{E}[ X(x) X(y) ] \: \partial_z \rho_{\epsilon}(z-x) \rho_{\epsilon}(z_1 - y)\\
  &= \int_{\mathbb{H}} \int_{\mathbb{H}} d^2x d^2y \: \partial_x \mathbb{E}[ X(x) X(y) ]  \rho_{\epsilon}(z-x) \rho_{\epsilon}(z_1 - y) \\
 &= -\frac{1}{2} \int_{\mathbb{H}} \int_{\mathbb{H}} d^2x d^2y  \left( \frac{1}{x-y}  + \frac{1}{x - \overline{y}} \right)  \rho_{\epsilon}(z-x) \rho_{\epsilon}(z_1 - y) - \frac{1}{2} \int_{\mathbb{H}} d^2x  \: \partial_x \ln \hat{g}(x)  \rho_{\epsilon}(z-x)\\
\mathbb{E}[\partial_z X_{\epsilon}(z) X_{\epsilon}(z)] &= \int_{\mathbb{H}} \int_{\mathbb{H}} d^2x d^2y \mathbb{E}[ X(x) X(y) ] \: \partial_z \rho_{\epsilon}(z-x) \rho_{\epsilon}(z - y)  \\
 &= \partial_z \frac{1}{2} \int_{\mathbb{H}} \int_{\mathbb{H}} d^2x d^2y   \ln \frac{1}{ \vert x - y  \vert}  \rho_{\epsilon}(z-x) \rho_{\epsilon}(z - y)\\
  &+ \int_{\mathbb{H}} \int_{\mathbb{H}} d^2x d^2y \: \partial_x \left( \ln \frac{1}{\vert x - \overline{y} \vert } - \frac{1}{2} \ln \hat{g}(x) \right) \rho_{\epsilon}(z-x) \rho_{\epsilon}(z - y)  \\
  &=   - \frac{1}{2} \int_{\mathbb{H}} \int_{\mathbb{H}} d^2x d^2y  \frac{1}{ x - \overline{y}  }  \rho_{\epsilon}(z-x) \rho_{\epsilon}(z - y)  - \frac{1}{2}  \int_{\mathbb{H}} d^2x \: \partial_x   \ln \hat{g}(x)  \rho_{\epsilon}(z-x)   \\
\mathbb{E}[\partial_z X_{\epsilon}(z) X_{\epsilon}(s)] &= 2 \int_{\mathbb{H}} \int_{\mathbb{H}} d^2x d^2y \mathbb{E}[ X(x) X(y) ] \: \partial_z \rho_{\epsilon}(z-x) \rho_{\epsilon}(s - y)  \\
 &= - \int_{\mathbb{H}} \int_{\mathbb{H}} d^2x d^2y \left( \frac{1}{ x- y } + \frac{1}{ x- \overline{y} } \right)  \rho_{\epsilon}(z-x) \rho_{\epsilon}(s - y) - \frac{1}{2} \int_{\mathbb{H}} d^2x  \: \partial_x \ln \hat{g}(x)  \rho_{\epsilon}(z-x)
\end{align*}
Putting everything together and taking $\epsilon \rightarrow 0$ we arrive at:
\begin{align*}
 \partial_z  \langle z_1, z \rangle &= \left( - \frac{ \gamma^2}{8}  \frac{1}{  z - \overline{z}}  + \frac{\gamma}{4}  \alpha ( \frac{1}{z - z_1} + \frac{1}{z - \overline{z_1}})   \right) \langle z, z_1 \rangle    -  \frac{  \mu_{\partial} \gamma^2}{4} \int_{\mathbb{R}} \frac{1}{z - s} \langle s,z, z_1 \rangle ds \\
 &+ \left( \frac{\gamma}{4} (\alpha - \frac{\gamma}{2} -Q)  \langle z, z_1 \rangle - \frac{\mu_{\partial} \gamma^2 }{8} \int_{\mathbb{R}} \langle s, z, z_1 \rangle ds \right)  \partial_z \ln \hat{g}(z) \\
 &= \left( - \frac{ \gamma^2}{8}  \frac{1}{  z - \overline{z}}  + \frac{\gamma}{4}  \alpha ( \frac{1}{z - z_1} + \frac{1}{z - \overline{z_1}})   \right) \langle z, z_1 \rangle    -  \frac{  \mu_{\partial} \gamma^2}{4} \int_{\mathbb{R}} \frac{1}{z - s} \langle s,z, z_1 \rangle ds. 
\end{align*}

To cancel the metric dependent terms in the last line we have used Lemma \ref{lem1}. The expression derived above shows that $(z_1, z) \mapsto \partial_z  \langle z_1, z \rangle $ is continuous on the set $\lbrace z_1,z \in \mathbb{H} \vert z_1 \neq z \rbrace$. Indeed, for $z, z_1$ in this set all the poles of the fractions appearing in $\partial_z  \langle z_1, z \rangle $ lie outside the set and the correlation functions can easily be shown to be continuous in their arguments. The derivatives $\partial_{z_1}$, $\partial_{\overline{z}} $,  $\partial_{\overline{z_1}}$, $\partial_{zz} $ and the other second order partial derivatives are computed along the same lines and from their expressions one obtains that they are also all continuous on the set $\lbrace z_1,z \in \mathbb{H} \vert z_1 \neq z \rbrace$. Hence one can deduce that $(z_1, z) \mapsto   \langle V_{\alpha}(z_1) V_{ -\frac{\gamma}{2}  }(z) \rangle_{\mathbb{H} } $ is $\mathcal{C}^2$ on $\lbrace z_1,z \in \mathbb{H} \vert z_1 \neq z \rbrace$. In the computations the only additional difficulty that appears is for $ \partial_{zz} \langle z, z_1 \rangle$, when applying the product rule to the expression of $ \partial_{z} \langle z, z_1 \rangle$ one needs an expression for $\partial_z \langle s, z, z_1 \rangle$. It is obtained by integration by parts just like \eqref{Gaussian_IPP}, but there is an extra term coming from the fact that $\partial_z \langle s, z, z_1 \rangle$ contains an extra insertion on the boundary. The analogue of \eqref{Gaussian_IPP} that one needs is:
\begin{align}\label{Gaussian_IPP2}
\langle X(f)  V_{\alpha}(z_1) V_{ -\frac{\gamma}{2}  }(z) V_{\gamma}(s) \rangle_{\mathbb{H}, \epsilon} &=  \alpha \mathbb{E}[X(f) X_{\epsilon}(z_1)] \langle s, z, z_1 \rangle_{ \epsilon } - \frac{\gamma}{2} \mathbb{E}[X(f) X_{\epsilon}(z)] \langle s, z, z_1 \rangle_{ \epsilon } \\
 &  + \frac{\gamma}{2} \mathbb{E}[X(f) X_{\epsilon}(s)] \langle s, z, z_1 \rangle_{ \epsilon } - \mu_{\partial} \frac{\gamma}{2} \int_{ \mathbb{R}} \mathbb{E}[X(f) X_{\epsilon}(t)] \langle s, t, z, z_1 \rangle_{ \epsilon } dt. \nonumber
\end{align}


\begin{proof}[Proof of Theorem~\ref{BPZequa}]
Following the method given above we compute all the derivatives that we need:

\begin{align}
\partial_{z_1}  \langle z, z_1 \rangle  & = \left(  \frac{\alpha \gamma}{4} ( \frac{1}{z_1 - z} + \frac{1}{z_1 - \overline{z}}  )  - \frac{\alpha^2}{2}    \frac{1}{z_1 - \overline{z_1}}   \right) \langle z, z_1 \rangle  +  \frac{ \alpha \mu_{\partial} \gamma}{2} \int_{ \mathbb{R}} \frac{1}{z_1 - s} \langle s,z, z_1\rangle ds \label{der1} \\
\partial_{\overline{z_1}}  \langle z, z_1 \rangle  & = \left(  \frac{\alpha \gamma}{4} ( \frac{1}{\overline{z_1} - z} + \frac{1}{ \overline{z_1} - \overline{z}}  )  - \frac{\alpha^2}{2}   \frac{1}{\overline{z_1} - z_1}   \right) \langle z, z_1 \rangle  +  \frac{ \alpha \mu_{\partial} \gamma}{2} \int_{ \mathbb{R}} \frac{1}{\overline{z_1} - s} \langle s, z, z_1 \rangle ds \label{der2}  \\
\partial_{z}  \langle z,z_1 \rangle  & = \left( - \frac{ \gamma^2}{8}  \frac{1}{  z - \overline{z}}  + \frac{\gamma}{4}  \alpha ( \frac{1}{z - z_1} + \frac{1}{z - \overline{z_1}})   \right) \langle z, z_1 \rangle    -  \frac{  \mu_{\partial} \gamma^2}{4} \int_{\mathbb{R}} \frac{1}{z - s} \langle s,z, z_1 \rangle ds \label{der3} \\
\partial_{\overline{z}}  \langle z,z_1 \rangle  & = \left( - \frac{ \gamma^2}{8}  \frac{1}{  \overline{z} -z}  + \frac{\gamma}{4} \alpha ( \frac{1}{\overline{z} - z_1} + \frac{1}{\overline{z} - \overline{z_1}})   \right) \langle z, z_1 \rangle   -  \frac{  \mu_{\partial} \gamma^2}{4} \int_{ \mathbb{R}} \frac{1}{ \overline{z}  - s} \langle s,z, z_1 \rangle ds \label{der4}
\end{align}
\begin{align}
\partial_{zz}  \langle z,z_1 \rangle & = \left(  \frac{ \gamma^2}{8}  \frac{1}{  (z - \overline{z})^2}  - \frac{\gamma}{4}  \alpha ( \frac{1}{(z - z_1)^2} + \frac{1}{(z - \overline{z_1})^2})   \right) \langle z, z_1 \rangle \label{der5} \\ 
 & + \left( \frac{\gamma^2}{8} \frac{1}{z -\overline{z}} - \frac{\gamma \alpha}{4} ( \frac{1}{z - z_1} + \frac{1}{z -\overline{z}_1})\right)^2  \langle z, z_1 \rangle   + \frac{  \mu_{\partial} \gamma^2}{4} \int_{ \mathbb{R}} \frac{1}{(z - s)^2} \langle s,z, z_1 \rangle ds \label{der6} \\
 & -\frac{\mu_{\partial} \gamma^3}{8} \left(  - \frac{\gamma}{2}  \frac{1}{  z - \overline{z}}  + \alpha ( \frac{1}{z - z_1} + \frac{1}{z - \overline{z_1}})   \right)   \int_{ \mathbb{R}} \frac{1}{z - s} \langle s,z, z_1 \rangle ds \label{der7} \\
& +  \frac{ \gamma^4 \mu_{\partial}^2 }{16} \int_{ \mathbb{R}} \int_{ \mathbb{R}} \frac{1}{z - s} \frac{1}{z - t} \langle s,t,z, z_1 \rangle ds dt  -  \frac{  \mu_{\partial} \gamma^4}{16} \int_{ \mathbb{R}} \frac{1}{(z - s)^2} \langle s,z, z_1 \rangle ds.\label{der8}
\end{align}
 We are now going to use these expressions we found for the partial derivatives of $\langle z, z_1 \rangle$ to check that the sum of all the terms of the BPZ equation add up to $0$. We start by checking that all the terms without  $\mu_{\partial}$ cancel correctly, we gather them based on their $\alpha$-dependence. Terms with $\alpha^2$:
\begin{align*}
\alpha^2 \bigg( - \frac{1}{4(z- z_1)^2} - \frac{1}{4(z-  \overline{z_1}  )^2} + \frac{1}{4} ( \frac{1}{z - z_1} + \frac{1}{z -  \overline{z_1}})^2 \quad & \\
- \frac{1}{2 (z - z_1) ( z_1 - \overline{z_1})  } + \frac{1}{2 (z - \overline{z_1}) ( z_1 - \overline{z_1})  } & \bigg) \langle z, z_1 \rangle   = 0.
\end{align*}
Terms with $ \alpha$:
\begin{align*}
& \alpha \left( (\frac{Q}{2} -  \frac{\gamma}{4} - \frac{1}{\gamma} ) ( \frac{1}{(z - z_1)^2} + \frac{1}{(z - \overline{z_1})^2} ) - \frac{\gamma }{4}  \frac{1}{z - \overline{z}} ( \frac{1}{z - z_1} + \frac{1}{z - \overline{z_1}} ) \right)\langle z, z_1 \rangle \\
 + & \alpha \left( \frac{\gamma}{4} ( \frac{1}{z - z_1} \frac{1}{z_1 - \overline{z}} + \frac{1}{z - \overline{z_1}} \frac{1}{\overline{z_1} - \overline{z}})   -  \frac{\gamma}{4} \frac{1}{z - \overline{z}} ( \frac{1}{z_1 - \overline{z}} + \frac{1}{\overline{z_1} - \overline{z}} ) \right) \langle z, z_1 \rangle =0.
\end{align*}
Terms with no $\alpha$:
$$ \left( \frac{1}{2} + \frac{\gamma^2}{8} + \frac{\gamma^2}{16} + \Delta_{-\frac{\gamma}{2}} \right) \frac{1}{(z - \overline{z})^2}  \langle z, z_1 \rangle   =0.  $$
We must now make sure that all the terms with $\mu_{\partial}$ cancel correctly, for this we need to perform an integration by parts on the last term of \eqref{der6}. However there is a slight subtlety coming from the fact that the derivative $\partial_s$ applied to $\langle s,z,z_1 \rangle$, computed using an analogue of \eqref{Gaussian_IPP2}, gives a term in $  \frac{1}{s-t}\langle s,t,z,z_1 \rangle$ and $ \frac{1}{s -t} $ is not integrable. But this difficulty can be easily overcome with our regularization procedure. We get, 
\begin{align}
  \frac{ \mu_{\partial} \gamma^2}{4} & \int_{\mathbb{R}}  \frac{1}{(z -s)^2}  \langle s,z,z_1 \rangle_{\epsilon} ds   =   \frac{ \mu_{\partial} \gamma^2}{4} \int_{\mathbb{R}} \partial_s (\frac{1}{z -s}) \langle s,z,z_1 \rangle_{\epsilon} ds  =  - \frac{ \mu_{\partial} \gamma^2}{4} \int_{\mathbb{R}}  \frac{1}{z -s} \partial_s \langle s,z,z_1 \rangle_{\epsilon} ds \nonumber  \\
  \underset{ \epsilon \rightarrow 0}{\longrightarrow}   & - \frac{ \mu_{\partial} \gamma^2}{4} \int_{\mathbb{R}} \frac{1}{z-s} \left( \frac{\gamma^2 }{4} ( \frac{1}{s - z} + \frac{1}{s - \overline{z}} ) - \frac{\gamma}{2}  \alpha ( \frac{1}{s - z_1} +  \frac{1}{s - \overline{z_1}} )    \right) \langle s, z, z_1 \rangle ds \label{der9}\\
  & -\frac{ \mu_{\partial}^2 \gamma^4}{8}  \lim_{\epsilon \rightarrow 0} \int_{ \mathbb{R}} \int_{ \mathbb{R}} \frac{1}{(z-s)} \frac{1}{(s - t)_{\epsilon,\epsilon}} \langle s, t, z, z_1 \rangle_{\epsilon} dt ds, \label{der10}
\end{align}
where we have introduced
$$ \frac{1}{(s - t)_{\epsilon,\epsilon}}  = 4 \int_{\mathbb{H}} \int_{ \mathbb{H}} d^2x_1 d^2x_2 \frac{1}{s + x_1 -t -x_2} \rho_{\epsilon}(x_1) \rho_{\epsilon}(x_2). $$
We symmetrize the last term:
\begin{align}\label{double2}
& - \frac{ \mu_{\partial}^2 \gamma^4}{8}  \int_{ \mathbb{R}} \int_{ \mathbb{R}}  \frac{1}{(z-s)} \frac{1}{(s - t)_{\epsilon,\epsilon}} \langle s, t, z, z_1 \rangle_{\epsilon} dt ds \nonumber  \\
& = - \frac{ \mu_{\partial}^2 \gamma^4}{16} \int_{ \mathbb{R}} \int_{ \mathbb{R}}  \frac{1}{(s - t)_{\epsilon,\epsilon}} (\frac{1}{z-s} - \frac{1}{z-t} ) \langle s, t, z, z_1 \rangle_{\epsilon} dt ds \nonumber \\
& = - \frac{ \mu_{\partial}^2 \gamma^4}{16}  \int_{ \mathbb{R}} \int_{ \mathbb{R}} \frac{1}{(z-s)(z-t) }  \frac{s-t}{(s - t)_{\epsilon,\epsilon}}  \langle s, t, z, z_1 \rangle_{\epsilon} dt ds \nonumber \\ 
& \underset{ \epsilon \rightarrow 0}{ \longrightarrow}  - \frac{\gamma^4 \mu_{\partial}^2}{16} \int_{\mathbb{R}} \int_{\mathbb{R}}  \frac{1}{z-s} \frac{1}{z-t}   \langle s, t, z, z_1 \rangle dt ds. 
\end{align}
From the above we see that the double terms in $\int_{\mathbb{R}} \int_{\mathbb{R}} $ coming from \eqref{der8} and \eqref{double2} cancel correctly. The terms in $\frac{1}{(z-s)^2} $  from \eqref{der8} and \eqref{der10} cancel. Finally we look at the cross terms, combining the integrals in \eqref{der4}, \eqref{der7}, \eqref{der9} containing a $\overline{z}$: 
\begin{align*}
\mu_{\partial} \frac{\gamma^2}{4} \int_{\mathbb{R}} \left( \frac{1}{z-s} \frac{1}{\overline{z} -s} + \frac{1}{z - \overline{z}} \frac{1}{z-s} - \frac{1}{\overline{z} -s } \frac{1}{z  - \overline{z}} \right) \langle s, z, z_1 \rangle ds = 0.
\end{align*}
Again combining the integrals in \eqref{der1}, \eqref{der7}, \eqref{der9} containing a $z_1$: 
\begin{align*}
\mu_{\partial} \frac{\alpha \gamma  }{2} \int_{\mathbb{R}} \left( \frac{1}{z-s} \frac{1}{s - z_1 } - \frac{1}{z - z_1} \frac{1}{z-s} + \frac{1}{z -z_1 } \frac{1}{z_1  - s} \right) \langle s, z, z_1 \rangle ds = 0.
\end{align*}
Lastly combining the integrals in \eqref{der2}, \eqref{der7}, \eqref{der9} containing a $\overline{z}_1$: 
\begin{align*}
\mu_{\partial} \frac{\alpha \gamma  }{2} \int_{\mathbb{R}} \left( \frac{1}{z-s} \frac{1}{s - \overline{z}_1 } - \frac{1}{z - \overline{z}_1} \frac{1}{z-s} + \frac{1}{z -\overline{z}_1 } \frac{1}{\overline{z}_1  - s} \right) \langle s, z, z_1 \rangle ds = 0.
\end{align*}
Therefore we have proved Theorem \ref{BPZequa}.
\end{proof}

\subsection{Correlation functions as moments of GMC on the unit circle}

We are now going to express our correlation function $\langle V_{-\frac{\gamma}{2}}(z)  V_{\alpha}(z_1)   \rangle_{\mathbb{H}}$ as a moment of Gaussian multiplicative chaos (GMC) on the unit circle and obtain from the BPZ equation of Theorem \ref{BPZequa} a differential equation on $G( \gamma, p ,t)$. 

\begin{proof}[Proof of Proposition~\ref{EquaG}]

As explained in appendix \ref{appmap}, the KPZ relation of \cite{Disk} tells us that we have the relation
\begin{equation}\label{KPZ}
 \langle V_{\alpha}(z_1) V_{- \frac{\gamma}{2}}(z) \rangle_{\mathbb{H}} =  \frac{1}{\vert z_1 - \overline{z}_1 \vert^{ 2 \Delta_{\alpha} - 2 \Delta_{- \frac{\gamma}{2}} }} \frac{1}{\vert z - \overline{z}_1 \vert^{ 4 \Delta_{- \frac{\gamma}{2}} }} \langle V_{\alpha}(0) V_{- \frac{\gamma}{2}}(t) \rangle_{\mathbb{D}}
\end{equation}
where $\Delta_{\alpha}$ and $\Delta_{- \frac{\gamma}{2}}$ are defined as in Theorem \ref{BPZequa} and where $\langle V_{-\frac{\gamma}{2}}(t)  V_{\alpha}(0)   \rangle_{\mathbb{D}}$ is the correlation of LCFT defined on the unit disk. As was observed in \cite{Disk}\footnote{See \cite[Corollary 3.10]{Disk}: our correlation function $ \langle V_{-\frac{\gamma}{2}}(t)  V_{\alpha}(0)   \rangle_{\mathbb{D}} $ is actually the normalization constant $\mathcal{Z}$ of the corollary, it can be computed by setting $F=1$. The idea is that once the correlation function has been conformally mapped to the unit disk $\mathbb{D}$,  \eqref{link1} is directly obtained from the analogue of \eqref{reg_partition} written on $\mathbb{D}$ by performing in the integral over $c$ the change variable $u  = \mu_{\partial} e^{\frac{\gamma c}{2}} \int_{ \partial \mathbb{D}} \epsilon^{\frac{\gamma^2}{4}} e^{\frac{\gamma}{2}X_{\epsilon}(s)} ds $ and sending $\epsilon$ to $0$. The integral over $u$ then gives the gamma function.} we can express it in terms of inverse moments of the GMC measure on the unit circle,
\begin{align}\label{link1}
 \langle V_{-\frac{\gamma}{2}}(t)  V_{\alpha}(0)   \rangle_{\mathbb{D}}  & = \frac{2}{\gamma} \mu_{\partial}^{ - \frac{2 \alpha - \gamma - 2Q }{\gamma} } \Gamma( \frac{2 \alpha - \gamma - 2Q }{\gamma} ) t^{ \frac{\alpha \gamma}{2} } ( 1 -t^2)^{ - \frac{\gamma^2}{8} }\\
 & \times  \mathbb{E}\left[ \left( \int_{\partial \mathbb{D}} e^{ \frac{\gamma}{2} (X(e^{i \theta}) - 2 \alpha \ln \vert e^{i \theta} \vert + \gamma \ln \vert e^{i \theta} -t \vert)} d\theta   \right)^{ - \frac{2 \alpha - \gamma - 2Q }{\gamma} } \right] \nonumber \\
 &  = \frac{2}{\gamma} \mu_{\partial}^{ - \frac{2 \alpha - \gamma - 2Q }{\gamma} } \Gamma( \frac{2 \alpha - \gamma - 2Q }{\gamma} ) t^{ \frac{\alpha \gamma}{2} } ( 1 -t^2)^{ - \frac{\gamma^2}{8} } G(\gamma,p,t),
\end{align}
where we have the following relation between our parameters $p$ and $\alpha$:
\begin{equation}\label{link2}
- p = \frac{2 \alpha - \gamma - 2Q }{\gamma} .
\end{equation}
The condition $\alpha > Q + \frac{\gamma}{2}$ of Theorem \ref{BPZequa} is then precisely the condition $ p< 0 $ of Proposition \ref{EquaG}. It is then a long but straightforward computation to turn the BPZ equation on $ \langle V_{\alpha}(z_1) V_{ - \frac{\gamma}{2}}(z) \rangle_{\mathbb{H}}$ into the differential equation on $G(\gamma,p,t)$, more details on this can be found in appendix \ref{appmap}. Therefore we have proved Proposition \ref{EquaG}.
\end{proof}

Let us make a few comments on our method. One might attempt to prove Proposition \ref{EquaG} without introducing at all the framework of Liouville theory on $\mathbb{H}$ and by just computing partial derivatives directly on the function $G(\gamma,p,t)$. In this computation  all terms cancel easily except a few terms coming from the second derivative in $t$ for which it is very difficult to see that they equal zero.\footnote{The term that one cannot see directly that it equals $0$ would be of the form, $$ \int_0^{2 \pi} \int_0^{2 \pi} du dv S(u,v,t) \mathbb{E} \left[ e^{\frac{\gamma}{2}X(e^{iu})} e^{\frac{\gamma}{2}X(e^{iv})} \left(\int_0^{2 \pi} \vert t - e^{i \theta} \vert^{\frac{\gamma^2}{2}} e^{\frac{\gamma}{2}X(e^{i \theta})} d \theta \right)^{p-2} \right],$$ for a complicated function $S$ of $u,v,t$. Showing this term cancels seems just as complicated as proving Theorem \ref{FyBo}.} Hence Liouville theory seems to be the correct framework where the computations are tractable and thus the KPZ relation \eqref{KPZ} - a highly non trivial change of variable - is a key ingredient of our proof. We mention that the same phenomenon appears for the BPZ equations used to prove the DOZZ formula in \cite{DOZZ1,DOZZ2} which are proved before applying the KPZ relation that fixes the automorphisms of the sphere. In other words it is not possible to permute the computation of derivatives and the uniformization given by \eqref{KPZ}. On the other hand our method has been very recently adapted in \cite{RZhu} to the case of GMC on the unit interval $[0,1]$ and there is great hope that it will also work on the two-dimensional GMC measure on the unit disk $\mathbb{D}$.  

\section{The shift equation for $U(\gamma,p)$}\label{sec_shift}

The goal of this section is to prove Proposition \ref{Shift} by identifying the coefficients $C_1$, $C_2$, $B_1$ and $B_2$ of Proposition \ref{CB}. For this purpose we will perform asymptotic expansions of $G(\gamma,p,t)$ in $t \rightarrow 0$ and in $ t \rightarrow 1$.

\subsection{Asymptotic expansion in $t \rightarrow 0$} 

Since $t \mapsto G( \gamma,p,t) $ is a continuous function on $[0,1]$, we have
$$ G(\gamma,p,0) = U( \gamma,p)$$
and since $p<0$, we cannot have a term in $t^{ \frac{\gamma^2}{2}(p-1) }$ in the expression $G( \gamma,p,t) $ and therefore we get:
\begin{equation}
C_2 = 0.
\end{equation}
Then taking $t =0$ in the expression of $G(\gamma,p,t)$, we find:
\begin{equation}\label{C1}
C_1 = U( \gamma, p).
\end{equation}

\subsection{Asymptotic expansion in $t \rightarrow 1$}
Our goal is now to perform the asymptotic expansion in $t \rightarrow 1$ which will give the following expression for the coefficient $B_2$:
\begin{equation}\label{B2}
B_2 = 2 \pi  p \frac{\Gamma( - \frac{\gamma^2}{2} -1)   }{\Gamma( - \frac{\gamma^2}{4} )^2 }  U( \gamma, p -1 ).
\end{equation}
We will  distinguish two cases based on whether $\gamma$ is smaller or greater than $\sqrt{2}$. The case $\gamma = \sqrt{2}$ will be handled by a continuity argument.
\subsubsection{The case $\gamma < \sqrt{2}$}
We start here with $\gamma \in (0, \sqrt{2})$. Taking $t=1$ in the expression of $G(\gamma,p,t)$, we get:
\begin{equation}
B_1  =  \mathbb{E}\left[ \left(\int_0^{2 \pi} \vert 1 - e^ {i \theta} \vert^{\frac{\gamma^2}{2}  } e^{ \frac{\gamma}{2} X(e^ {i \theta})} d \theta \right)^p \right].
\end{equation}
At this stage there is nothing we can do with this coefficient but as an output of our proof we will also obtain a value for this quantity, see Corollary \ref{corol}. We must now go to the next order to find $B_2$. We introduce the notation $ h_u(t) = \vert t - e^{iu} \vert^{ \frac{\gamma^2}{2} } $. In the following computations we will extensively use the Girsanov theorem (also called the Cameron-Martin formula) in the following way: 
\begin{align}\label{girsavov}
\mathbb{E}\left[\int_0^{2 \pi}  e^{\frac{\gamma}{2} X(e^ {i u}) } du \left(\int_0^{2 \pi} e^{ \frac{\gamma}{2} X(e^ {i \theta}) } d\theta \right)^{p-1} \right] &= \int_0^{2 \pi} \mathbb{E}\left[  \left(\int_0^{2 \pi} e^{ \frac{\gamma}{2} (X(e^ {i \theta}) + \frac{\gamma}{2} \mathbb{E}[X(e^{i \theta})X(e^{iu})] ) }  d\theta \right)^{p-1} \right] du \nonumber \\
 &= \int_0^{2 \pi} \mathbb{E}\left[  \left(\int_0^{2 \pi} \frac{e^{ \frac{\gamma}{2} X(e^ {i \theta}) }}{\vert e^{iu} - e^{i \theta } \vert^{ \frac{\gamma^2}{2} }   }  d\theta \right)^{p-1} \right] du.
\end{align}
To derive \eqref{girsavov} rigorously one needs to perform the computation with $X$ replaced by an approximation $X_{\epsilon}$ and then take $\epsilon$ to $0$. We then write, using Taylor's formula at the first step and \eqref{girsavov} in the second:
\begin{align}\label{ref_dl}
\mathbb{E}& \left[  \left(  \int_{0}^{2 \pi} \vert t - e^{i \theta} \vert^{  \frac{\gamma^2}{2}} e^{ \frac{\gamma}{2} X( e^ {i \theta}) } d \theta \right)^p \right] -  \mathbb{E}\left[  \left( \int_{0}^{2 \pi} \vert 1 - e^{i \theta} \vert^{  \frac{\gamma^2}{2}} e^{ \frac{\gamma}{2} X(e^ {i \theta}) } d \theta \right)^{p  }\right] \nonumber \\
& =  p  \mathbb{E}\left[   \int_0^{2 \pi} du \left( \vert t - e^{i u} \vert^{ \frac{\gamma^2}{2}} -\vert 1 - e^{i u} \vert^{ \frac{\gamma^2}{2}} \right) e^{ \frac{\gamma}{2} X(e^ {i u})}  \left( \int_0^{2 \pi} \vert 1 - e^{i \theta} \vert^{ \frac{\gamma^2}{2}} e^{ \frac{\gamma}{2} X( e^ {i \theta})} d \theta \right)^{ p-1 }\right]  + R(t) \nonumber  \\
& =  p  \int_0^{2 \pi} du ( h_u(t) - h_u(1) ) \mathbb{E}\left[ \left( \int_0^{2 \pi} \frac{\vert 1 - e^{i \theta}    \vert^{ \frac{\gamma^2}{2}} }{\vert e^{iu} - e^{i \theta}    \vert^{ \frac{\gamma^2}{2}}}   e^{ \frac{\gamma}{2} X( e^ {i \theta})} d \theta \right)^{ p-1  }\right]  + R(t). 
\end{align}
$R(t)$ are higher order terms that are given by the Taylor formula applied to $ x \mapsto x^p$ :
\begin{align*}
 R(t) & = p(p-1) \int_0^{2 \pi} \int_0^{2 \pi} d\theta_1 d\theta_2 (h_{\theta_1}(t) - h_{\theta_1}(1) ) (h_{\theta_2}(t) - h_{\theta_2}(1) ) \\ 
& \times \mathbb{E} \left[  \int_{s=0}^1 ds (1 -s) \left( \int_0^{2 \pi} d \theta \frac{h_{\theta}(1) + s(h_{\theta}(t) -h_{\theta}(1)  ) } {\vert e^{i\theta_1} - e^{i\theta} \vert^{\frac{\gamma^2}{2} } \vert e^{i\theta_2} - e^{i\theta} \vert^{\frac{\gamma^2}{2} } }  e^{\frac{\gamma}{2} X(e^ {i \theta}) } \right)^{p-2} \right]. 
\end{align*}
One may wonder if the GMC measures with fractional powers that appear in the above computations are well-defined. It is known that for $\beta \in \mathbb{R}$, 
$$\int_0^{2 \pi} \frac{1}{\vert 1 - e^{i \theta}\vert^{\frac{\beta \gamma}{2}} } e^{\frac{\gamma}{2} X(e^{i \theta})} d \theta < + \infty \: \: \:  a.s. \Leftrightarrow \beta < \frac{\gamma}{2} + \frac{2}{\gamma}, $$
see for instance \cite[Lemma 2.7]{review2}. Therefore in the expression of $R(t)$ the only problem is when $\theta_1 = \theta_2 $. But in our case we are dealing with negative moments so at $\theta_1 = \theta_2$ we simply get $0$. It is thus straightforward to bound $R(t)$ by
\begin{align}\label{bound_R}
\vert R(t) \vert &  \leq M \int_0^{2 \pi} \int_0^{2 \pi} d\theta_1  d\theta_2 (h_{\theta_1}(t) - h_{\theta_1}(1) ) (h_{\theta_2}(t) - h_{\theta_2}(1) )  \\
& \leq \tilde{M} (1 -t)^2 \nonumber
\end{align}
for some $M, \tilde{M}>0$. Coming back to the first term of  \eqref{ref_dl}, notice that the following integral can be expressed in terms of hypergeometric functions $F$:
\begin{align}\label{int1}
\frac{1}{2 \pi} \int_0^{2 \pi}  (h_u(t) - h_u(1) ) du  & = F( -\frac{\gamma^2}{4}, -\frac{\gamma^2}{4}, 1,t^2 ) - F( -\frac{\gamma^2}{4}, -\frac{\gamma^2}{4}, 1,1 ) \\
& = \frac{\Gamma( \frac{\gamma^2}{2} + 1 )}{\Gamma( \frac{\gamma^2}{4} +1 )^2}  F( -\frac{\gamma^2}{4}, -\frac{\gamma^2}{4}, -\frac{\gamma^2}{2}, 1 -t^2 ) - F( -\frac{\gamma^2}{4}, -\frac{\gamma^2}{4}, 1,1 ) \nonumber  \\
& +   \frac{\Gamma(- \frac{\gamma^2}{2} -1 )}{\Gamma(- \frac{\gamma^2}{4})^2} (1 -t^2)^{1 + \frac{\gamma^2}{2}}F( 1 +\frac{\gamma^2}{4},1 +\frac{\gamma^2}{4}, 2 + \frac{\gamma^2}{2}, 1 -t^2 ).  \nonumber
\end{align}
In the last line we have used the formula \eqref{hpy2} given in appendix \ref{apphyp} valid here for $\gamma \neq \sqrt{2} $. We notice that $1 < 1 + \frac{\gamma^2}{2}< 2 $ and that \eqref{int1} implies:
\begin{equation}\label{DL1}
\lim_{t \rightarrow 1} \frac{1}{(1 -t^2)^{1 + \frac{\gamma^2}{2}}} \frac{1}{2 \pi} \int_0^{2 \pi} \left(h_u(t) - h_u(1) - (t-1)h'_u(1) \right) du = \frac{\Gamma(-\frac{\gamma^2}{2} -1)   }{\Gamma(-\frac{\gamma^2}{4} )^2}.
\end{equation}
The key observation is that the same result holds if we add some continuous function $c$ defined on the unit circle:

\begin{lemma}\label{LEM}
For $\epsilon > 0$ and $u \in [0, 2 \pi] $ let $ h_u(1 - \epsilon) = \vert 1 - \epsilon - e^{iu} \vert^{ \frac{\gamma^2}{2} } $ and let $c: \partial \mathbb{D} \mapsto \mathbb{R}$ be a continuous function defined on the unit circle. Then we have:
\begin{equation}\label{Lem}
\lim_{\epsilon \rightarrow 0} \frac{1}{\epsilon^{1 + \frac{\gamma^2}{2}}} \frac{1}{2^{1 + \frac{\gamma^2}{2}} } \frac{1}{2 \pi} \int_0^{2 \pi} \left( h_u(1 - \epsilon) - h_u(1) + \epsilon h'_u(1) \right) c(e^{i u }) du = \frac{\Gamma(-\frac{\gamma^2}{2} -1)   }{\Gamma(-\frac{\gamma^2}{4} )^2} c(1).
\end{equation}

\end{lemma}

\begin{proof}

We start by showing that
$$ \frac{1}{\epsilon^{1 + \frac{\gamma^2}{2}}}  \int_0^{2 \pi} \vert h_u(1 - \epsilon) - h_u(1) + \epsilon h'_u(1) \vert du$$
remains bounded as $\epsilon$ goes to $0$. We split the integral into two parts, $\int_{- \epsilon}^{\epsilon} $ and $\int_{ \epsilon }^{2 \pi - \epsilon} $. To analyse the first part we can perform an asymptotic expansion on $ u \mapsto e^{iu} $, we get:
\begin{align*}
 & \frac{1}{\epsilon^{1 + \frac{\gamma^2}{2}}} \int_{ - \epsilon  }^{ \epsilon}  \vert h_u(1 - \epsilon) - h_u(1) + \epsilon h'_u(1) \vert du \\
& \leq  \frac{M_1}{\epsilon^{1 + \frac{\gamma^2}{2}}}  \int_{ - \epsilon  }^{ \epsilon} \big \vert  \vert \epsilon + iu \vert^{ \frac{\gamma^2}{2} }   - \vert u \vert^{ \frac{\gamma^2}{2} } + \epsilon \frac{\gamma^2}{4} \vert u \vert^{ \frac{\gamma^2}{2} }  \big \vert du \\
 & = \frac{M_1}{\epsilon^{1 + \frac{\gamma^2}{2}}}  \int_{ -1  }^{1} \big \vert  \vert \epsilon + i\epsilon v \vert^{ \frac{\gamma^2}{2} }   - \vert \epsilon v  \vert^{ \frac{\gamma^2}{2} } + \epsilon \frac{\gamma^2}{4} \vert \epsilon v  \vert^{ \frac{\gamma^2}{2} }   \big \vert \epsilon dv \leq M_2
\end{align*}
for some $M_1, M_2 >0$. The other part of the integral can be bounded by the Taylor formula:
\begin{align*}
\frac{1}{\epsilon^{1 + \frac{\gamma^2}{2}}}  \int_{\epsilon}^{2 \pi - \epsilon} \vert h_u(1 - \epsilon) - h_u(1) + \epsilon h'_u(1) \vert du & \leq \frac{1}{\epsilon^{1 + \frac{\gamma^2}{2}}}  \int_{\epsilon}^{2 \pi - \epsilon} \frac{\epsilon^2}{2} \sup_{ x \in [1-\epsilon,1]} \vert h''_u(x)  \vert du \\
  & \leq M_3  \epsilon^{ 1 - \frac{\gamma^2}{2} } \int_{\epsilon}^1 \vert u \vert^{ \frac{\gamma^2}{2} -2  } du  \leq M_4
\end{align*}
for some constants $M_3,M_4 >0$. By continuity of $c$, for $ \epsilon' >0$ there exists an $\eta > 0$ such that $\forall u \in (-\eta, \eta) $, $ \vert c(e^{iu}) - c(1) \vert \leq \epsilon'$. We can then write:
\begin{align*}
& \frac{1}{\epsilon^{1 + \frac{\gamma^2}{2}}}  \int_0^{2 \pi} \vert h_u(1 - \epsilon) - h_u(1) + \epsilon h'_u(1) \vert \vert c(e^{iu}) - c(1) \vert du \\
& = \frac{1}{\epsilon^{1 + \frac{\gamma^2}{2}}}  \int_{ - \eta  }^{ \eta } \vert h_u(1 - \epsilon) - h_u(1) + \epsilon h'_u(1) \vert \vert c(e^{iu}) - c(1) \vert du \\ 
& + \frac{1}{\epsilon^{1 + \frac{\gamma^2}{2}}}  \int_{ \eta}^{ 2 \pi -\eta}   \vert h_u(1 - \epsilon) - h_u(1) + \epsilon h'_u(1) \vert \vert c(e^{iu}) - c(1) \vert du \\
& \leq \tilde{M}_1 \epsilon' + \frac{1}{\epsilon^{1 + \frac{\gamma^2}{2} }} \epsilon^2 \tilde{M}_2  
\end{align*}
for some $\tilde{M}_1, \tilde{M}_2 >0$. Since the above is true for all $\epsilon'$ we easily arrive at:
$$  \lim_{\epsilon \rightarrow 0} \frac{1}{\epsilon^{1 + \frac{\gamma^2}{2}}}  \int_0^{2 \pi} \vert h_u(1 - \epsilon) - h_u(1) + \epsilon h'_u(1) \vert \vert c(e^{iu}) - c(1) \vert du =0.$$
From this and using the exact computation \eqref{DL1} we obtain \eqref{Lem}.
\end{proof}

We then apply the Lemma \ref{LEM} to our problem by choosing:
$$ c(e^{iu}) = \mathbb{E}\left[ \left( \int_0^{2 \pi} \frac{\vert 1 - e^{i \theta}    \vert^{ \frac{\gamma^2}{2}} }{\vert e^{iu} - e^{i \theta}    \vert^{ \frac{\gamma^2}{2}}}   e^{ \frac{\gamma}{2} X( e^ {i \theta})} d \theta \right)^{ p-1  } \right]. $$
This provides an asymptotic expansion for the first term of \eqref{ref_dl} up to the order $o((1-t)^{1 + \frac{\gamma^2}{2}}) $. Combined with the bound \eqref{bound_R} found for $R(t)$, and recalling here that $\frac{\gamma^2}{2} <2$, one obtains the following expansion in powers of $(1-t)$ for $G(\gamma,p,t)$:
\begin{align*}
G(\gamma, p , t ) &= B_1 + p( t-1) \int_0^{2 \pi} du h'_u(1) \mathbb{E} \left[ \left( \int_0^{2 \pi}  \frac{\vert 1 - e^{i \theta} \vert^{\frac{\gamma^2}{2}   } }{\vert e^{iu} - e^{i \theta} \vert^{\frac{\gamma^2}{2} }  } e^{ \frac{\gamma}{2} X( e^ {i \theta}) d \theta}  \right)^{p-1} \right] \\
&+ 2  \pi  p \frac{\Gamma( - \frac{\gamma^2}{2} -1)   }{\Gamma( - \frac{\gamma^2}{4} )^2 }  U( \gamma, p -1 ) (1 - t^2)^{ 1 + \frac{\gamma^2}{2} } + o((1 - t)^{ 1 + \frac{\gamma^2}{2} }).
\end{align*}
This gives the annouced value for $B_2$ stated in \eqref{B2}.

\subsubsection{The case $\gamma > \sqrt{2}$}
In order to repeat the argument given above in the case $ \gamma \in (\sqrt{2},2)$, we need to go one order further in the computations. If we go one order further in \eqref{DL1} we still get:
\begin{equation}\label{DL2}
\lim_{t \rightarrow 1} \frac{1}{(1 -t^2)^{1 + \frac{\gamma^2}{2}}} \frac{1}{2 \pi} \int_0^{2 \pi} (h_u(t) - h_u(1) - (t-1)h'_u(1) - \frac{1}{2}(t-1)^2 h''_u(1) ) du = \frac{\Gamma(-\frac{\gamma^2}{2} -1)   }{\Gamma(-\frac{\gamma^2}{4} )^2}.
\end{equation}
Lemma \ref{LEM} still holds if we go one order further and the analysis of $R(t)$ gives this time
\begin{equation}\label{DL3}
R(t)  =  m(1-t)^2 + O((1-t)^3)
\end{equation}
for some $m \in \mathbb{R}$. Indeed since we can write
$$h_{\theta_1}(t) - h_{\theta_1}(1) = (t-1)h'_{\theta_1}(1) + \frac{(t-1)^2}{2} h''_{\theta_1}(1) + o((t-1)^2), $$
and we have
\begin{align*}
R(t) &= p(p-1)(t-1)^2 \int_0^{2 \pi} \int_0^{2 \pi}  h'_{\theta_1}(1) h'_{\theta_2}(1)\\
&\times  \mathbb{E} \left[ \int_{s=0}^1 ds (1 -s) \left( \int_0^{2\pi} d\theta \frac{h_{\theta}(1) + s(h_{\theta}(t) -h_{\theta}(1)  ) } {\vert e^{i\theta_1} - e^{i\theta} \vert^{\frac{\gamma^2}{2} } \vert e^{i\theta_2} - e^{i\theta} \vert^{\frac{\gamma^2}{2} } }  e^{\frac{\gamma}{2} X(e^ {i \theta}) } \right)^{p-2} \right] +O((t-1)^3),
\end{align*}
applying one last Taylor expansion on the above expectation
 $\mathbb{E}[.]$ we get
 \begin{align*}
 &\mathbb{E} \left[ \int_{s=0}^1 ds (1 -s) \left( \int_0^{2\pi} d\theta \frac{h_{\theta}(1) + s(h_{\theta}(t) -h_{\theta}(1)  ) } {\vert e^{i\theta_1} - e^{i\theta} \vert^{\frac{\gamma^2}{2} } \vert e^{i\theta_2} - e^{i\theta} \vert^{\frac{\gamma^2}{2} } }  e^{\frac{\gamma}{2} X(e^ {i \theta}) } \right)^{p-2} \right] \\
 &  = \frac{1}{2}\mathbb{E} \left[  \left( \int_0^{2\pi} d\theta \frac{h_{\theta}(1)   e^{\frac{\gamma}{2} X(e^ {i \theta}) } } {\vert e^{i\theta_1} - e^{i\theta} \vert^{\frac{\gamma^2}{2} } \vert e^{i\theta_2} - e^{i\theta} \vert^{\frac{\gamma^2}{2} } }  \right)^{p-2} \right] + O(t-1),
 \end{align*}
and so we finally arrive at \eqref{DL3}. 
From the above we see that we can write an expansion of $G(\gamma,p,t)$ of the form,
$$ G(\gamma, p,t) = B_1 + b_1 (1-t) + b_2 (1 -t)^2 + 2 \pi  p \frac{\Gamma( - \frac{\gamma^2}{2} -1)   }{\Gamma( - \frac{\gamma^2}{4} )^2 }  U( \gamma, p -1 ) (1-t^2)^{ 1 + \frac{\gamma^2}{2} } + o((1 -t)^{1 + \frac{\gamma^2}{2}}), $$
for some $b_1, b_2 \in \mathbb{R}$. From this we deduce \eqref{B2} in the case $\sqrt{2} < \gamma < 2$.

\begin{proof}[Proof of Proposition~\ref{Shift}]
Starting from Proposition \ref{CB}, the change of basis formula \eqref{hpy1} coming from the theory of hypergeometric functions combined with the fact that $C_2 =0$ leads for $p<0$ and $\gamma \neq \sqrt{2}$ to the relation:
\begin{equation}
B_2 = \frac{\Gamma(-1 - \frac{\gamma^2}{2}  ) \Gamma( \frac{\gamma^2}{4}(1 -p) +1  )  }{\Gamma( - \frac{\gamma^2}{4}  )  \Gamma( - \frac{\gamma^2 p}{4}  )} C_1.
\end{equation}
Now using the expressions \eqref{C1} and \eqref{B2} found for $C_1$ and $B_2$ we arrive at for $p<0$ and $\gamma \neq \sqrt{2}$:
\begin{equation}\label{shift3}
\frac{U(\gamma,p ) }{U( \gamma, p-1)} = \frac{2 \pi p \Gamma(- \frac{\gamma^2 p}{4} )}{\Gamma(-\frac{\gamma^2}{4} ) \Gamma(1 - (p-1) \frac{\gamma^2}{4} )} = \frac{2 \pi \Gamma(1 - p \frac{\gamma^2}{4} ) }{ \Gamma(1 - \frac{\gamma^2}{4}) \Gamma(1 - (p-1) \frac{\gamma^2}{4} )}.
\end{equation}
By continuity of $ p \mapsto U(\gamma,p) $ we can take the limit $p \rightarrow 0$ in the above relation to get the shift equation for all $p \leq0$. Then by continuity of $ \gamma \mapsto U(\gamma,p)$, a simple exercise, we extend \eqref{shift3} to the case $\gamma = \sqrt{2}$. Therefore we have proved Proposition \ref{Shift}.
\end{proof}

\section{Appendix}

\subsection{Mapping of the BPZ equation to the unit disk}\label{appmap}

In this section we turn the BPZ equation we found on $\mathbb{H}$ using Liouville theory into an equation on the function $G(\gamma,p,t)$. We recall that the conformal map $ \psi_1: x \mapsto \frac{x-i}{x+i}$ maps the upper half plane $\mathbb{H}$ equipped with metric $\hat{g}(x) = \frac{4}{\vert x + i \vert^4} $ to the unit disk $\mathbb{D}$ equipped with the Euclidean metric. The KPZ formula of \cite[Proposition 3.7]{Disk} for a change of domain then tells us that:
\begin{equation}
\langle V_{\alpha}(z_1) V_{- \frac{\gamma}{2}}(z) \rangle_{\mathbb{H}} =  \vert \psi'_1(z_1) \vert^{ 2 \Delta_{\alpha}  } \vert \psi'_1(z) \vert^{ 2 \Delta_{- \frac{\gamma}{2}} } \langle V_{\alpha}(\psi_1(z_1)) V_{- \frac{\gamma}{2}}(\psi_1(z)) \rangle_{\mathbb{D}}.
\end{equation}
We apply again the KPZ formula of \cite[Theorem 3.5]{Disk} to the conformal map $\psi_2$ of $\mathbb{D}$ onto $\mathbb{D}$ that maps $\psi_1(z_1)$ to $0$ and $\psi_1(z)$ to $t \in (0,1)$. More explicitly $\psi_2(x) = e^{i \theta} \frac{x - \psi_1(z_1)}{1 - x \overline{\psi_1(z_1)}} $ with $\theta$ chosen such that $ \psi_2(\psi_1(z)) \in (0,1)$. This gives:
\begin{equation}
\langle V_{\alpha}(\psi_1(z_1)) V_{- \frac{\gamma}{2}}(\psi_1(z)) \rangle_{\mathbb{D}} = \vert \psi'_2(\psi_1(z_1)) \vert^{ 2 \Delta_{\alpha}  } \vert \psi'_2(\psi_1(z)) \vert^{ 2 \Delta_{- \frac{\gamma}{2}} } \langle V_{\alpha}(0) V_{- \frac{\gamma}{2}}(t) \rangle_{\mathbb{D}}.
\end{equation}
Combining both of the above relations we arrive at:
\begin{equation}
\langle V_{\alpha}(z_1) V_{- \frac{\gamma}{2}}(z) \rangle_{\mathbb{H}} =  \frac{1}{\vert z_1 - \overline{z}_1 \vert^{ 2 \Delta_{\alpha} - 2 \Delta_{- \frac{\gamma}{2}} }} \frac{1}{\vert z - \overline{z}_1 \vert^{ 4 \Delta_{- \frac{\gamma}{2}} }} \langle V_{\alpha}(0) V_{- \frac{\gamma}{2}}(t) \rangle_{\mathbb{D}}.
\end{equation}
Now starting from our BPZ equation,
\begin{equation*}
(\frac{4}{\gamma^2} \partial_{zz} + \frac{ \Delta_{- \frac{\gamma}{2}}  }{(z - \overline{z})^2} + \frac{\Delta_{\alpha}}{ (z - z_1)^2} + \frac{\Delta_{\alpha}}{ (z - \overline{z}_1)^2} + \frac{1}{z - \overline{z}} \partial_{ \overline{z}} + \frac{1}{z - z_1} \partial_{ z_1}  + \frac{1}{z - \overline{z_1}} \partial_{ \overline{z_1}} )  \langle V_{\alpha}(z_1) V_{- \frac{\gamma}{2}}(z) \rangle_{\mathbb{H}} =0,
\end{equation*}
we obtain the following equation for $\langle V_{\alpha}(0) V_{- \frac{\gamma}{2}}(t) \rangle_{\mathbb{D}} $:
\begin{equation}
\left( \frac{t^2}{\gamma^2} \frac{d^2}{dt^2} + \left(- \frac{ t}{\gamma^2} + \frac{2 t^3 - t}{2(1-t^2)}\right) \frac{d}{dt} + \left(\Delta_{\alpha} + \Delta_{ - \frac{\gamma}{2}}   \frac{2 t^2 -t^4}{( t^2 -1)^2}  \right) \right)\langle V_{\alpha}(0) V_{- \frac{\gamma}{2}}(t) \rangle_{\mathbb{D}} = 0.
\end{equation}
Then using \eqref{link1} and \eqref{link2}, we get the announced differential equation of Proposition \ref{EquaG} for $G(\gamma,p,t)$,
\begin{equation}
( t (1 - t^2) \frac{\partial^2}{\partial t^2} +(t^2-1)\frac{\partial}{\partial t} + 2( C - (A +B+1)t^2) \frac{\partial}{\partial t} - 4ABt ) G(\gamma,p, t) = 0,
\end{equation}
with $ A = -\frac{\gamma^2 p}{4}$, $B = - \frac{\gamma^2}{4}$, and $ C = \frac{\gamma^2}{4}(1 - p) + 1  $.

\subsection{Hypergeometric functions}\label{apphyp}

We recall here briefly all the fact on hypergeometric functions that we have used throughout our paper. For $A$, $B$, $C$, and $x$ real numbers we introduce the following power series,
\begin{equation}
F(A,B,C,x) = \sum_{n=0}^{ \infty} \frac{A_n B_n}{ n! C_n} x^n,
\end{equation}
where $ A_n = \frac{\Gamma(A +n)}{\Gamma(A)}$ for $ n \in \mathbb{N}$, $B_n$ and $C_n$ are defined similarly with $B$ or $C$ in place of $A$, and $\Gamma$ is the standard gamma function. The function $F$ is the hypergeometric series and it can be used to solve the following hypergeometric equation:
\begin{equation}
\left( x(1-x) \frac{\partial^2}{\partial x^2} + (C -(A +B +1)x) \frac{\partial}{\partial x} - AB \right) H(x) = 0.
\end{equation}
Under the condition that $C$ and $C-A-B$ are not integers, the solutions of this equation can be given in two different bases, one corresponding to an expansion in powers of $x$ and the other in powers of $1 -x$. We write:
\begin{align}
H(x) & = C_1 F(A,B,C,x) + C_2 x^{1-C} F(1 +A-C,1 +B -C,2-C,x), \\
H(x) & = B_1 F(A,B,1 +A +B -C,1 - x) + B_2 (1-x)^{C-A -B} F(C-A,C-B,1 + C - A -B,1 - x). \nonumber
\end{align}
In our case where $ A = -\frac{\gamma^2 p}{4}$, $B = - \frac{\gamma^2}{4}$, and $ C = \frac{\gamma^2}{4}(1 - p) + 1  $, the four real constants $C_1$, $C_2$, $B_1$, $B_2$ are linked by the following relation:
\begin{equation}\label{hpy1}
\begin{pmatrix}
B_1  \\
B_2
\end{pmatrix} 
= 
 \begin{pmatrix}
\frac{\Gamma(1 + \frac{\gamma^2}{2}) \Gamma( \frac{\gamma^2}{4}(1 -p) +1  )  }{\Gamma(1 + \frac{\gamma^2}{4})\Gamma( \frac{\gamma^2}{4}(2 -p) +1  )  } & \frac{\Gamma(1 + \frac{\gamma^2}{2}) \Gamma( \frac{\gamma^2}{4}(p-1) +1 )  }{\Gamma(1 + \frac{\gamma^2}{4})\Gamma( \frac{\gamma^2}{4}p +1 )   } \\
 \frac{\Gamma(-1 - \frac{\gamma^2}{2}  ) \Gamma( \frac{\gamma^2}{4}(1 -p) +1  )  }{\Gamma( - \frac{\gamma^2}{4}  )  \Gamma( - \frac{\gamma^2 p}{4}  )} & \frac{\Gamma( -1 - \frac{\gamma^2}{2} )  \Gamma( \frac{\gamma^2}{4}(p-1) +1) }{ \Gamma(  - \frac{\gamma^2}{4} ) \Gamma( \frac{\gamma^2}{4}(p-2) )  }
\end{pmatrix} 
\begin{pmatrix}
C_1  \\
C_2 
\end{pmatrix}.
\end{equation}
This relation can easily be deduced from the following properties of the hypergeometric function:
\begin{align}\label{hpy2}
F(A,B,C,x) & = \frac{\Gamma(C)\Gamma(C -A - B)}{\Gamma(C-A) \Gamma(C-B)} F(A,B,A+B-C+1,1-x) \\
& + (1 -x )^{C-A-B}\frac{\Gamma(C)\Gamma(A +B -C )}{\Gamma(A) \Gamma(B)} F(C-A,C-B,C-A-B+1,1-x) \nonumber
\end{align}
and
\begin{equation}
F(A,B,C,1) = \frac{\Gamma(C) \Gamma(C-A-B)}{\Gamma(C-A) \Gamma(C-B)}.
\end{equation}
More details on hypergeometric functions can be found in \cite{Hyp}.


\begin{thebibliography}{20}


\bibitem{Hyp} Andrews G.E., Askey R., Roy R.: Special functions, \emph{Enclyclopedia of Mathematics and its Applications}, volume 71, (1999).

\bibitem{Aru} Aru J., Huang Y., Sun X.: Two perspectives of the 2D unit area quantum sphere and their equivalence, \emph{Commun. Math. Phys.}, 356:261 (2017).


\bibitem{APS}
Aru J., Powell E., Sep\'ulveda A.: Critical Liouville measure as a limit of subcritical measures, to appear in \emph{Electron. Commun. Probab.}, arXiv:1802.08433. 

\bibitem{Arguin} Arguin L.P.: Extrema of log-correlated random variables: principles and examples, arXiv:1601.00582.

\bibitem{ArBeBour}
Arguin L.P., Belius D., Bourgade P.: Maximum of the Characteristic Polynomial of Random Unitary Matrices, \emph{Commun. Math. Phys.} January 2017, Volume 349, Issue 2, pp 703-751.
 
\bibitem{BPZ} Belavin A.A., Polyakov A.M., Zamolodchikov A.B.: Infinite conformal symmetry in two-dimensional quantum field theory,  \emph{Nuclear. Physics.}, B241, 333-380, (1984). 

\bibitem{Ber} Berestycki N.: An elementary approach to Gaussian multiplicative chaos, \emph{Electron. Commun. Probab.} 22 (2017), paper no. 27, 12 pp. doi:10.1214/17-ECP58.

\bibitem{BiLoui}
Biskup M., Louidor O.: Extreme local extrema of two-dimensional discrete Gaussian free field, \emph{Commun. Math. Phys.} 345, no. 1 (2016): 271-304.

\bibitem{ChMaNaj}
Chhaibi R., Madaule T., Najnudel J.: On the maximum of the C$\beta$E field, \emph{Duke Math. J.} 167 (2018), no. 12, 2243--2345. doi:10.1215/00127094-2018-0016.

\bibitem{ChNaj} Chhaibi R., Najnudel J.: On the circle, $GMC^{\gamma} = C \beta E_{\infty}$ for $\gamma = \sqrt{\frac{2}{\beta}} \quad (\gamma \leq 1)$, arXiv:1904.00578.

\bibitem{DiRoZei} Ding J., Roy R., Zeitouni O.: Convergence of the centered maximum of log-correlated Gaussian fields, \emph{Annals of Probability}, Volume 45, Number 6A, 3886-3928, (2017).



\bibitem{Tori}
David F., Rhodes R., Vargas V.: Liouville quantum gravity on complex tori, \emph{J. Math. Phys.}, \textbf{57}: 022302 (2016). 

\bibitem{Sphere} 
David F., Kupiainen A., Rhodes R., Vargas V.: Liouville quantum gravity on the Riemann sphere, \emph{Commun. Math. Phys.}, \textbf{342}: 869-907 (2016).


\bibitem{DO}
Dorn H., Otto H.-J.: Two and three point functions in Liouville theory, \emph{Nuclear Physics B}, \textbf{429} (2), 375-388 (1994). 

\bibitem{Mating} Duplantier B., Miller J., Sheffield S.: Liouville quantum gravity as a mating of trees, arXiv:1409.7055.
 
\bibitem{DRSV1}
Duplantier  B., Rhodes R., Sheffield S., Vargas V.: Critical Gaussian multiplicative chaos: convergence of the derivative martingale, \emph{Annals of Probability} vol  42, Number 5 (2014), 1769-1808.


\bibitem{DRSV2}
Duplantier  B., Rhodes R., Sheffield S., Vargas V.: Renormalization of Critical Gaussian Multiplicative Chaos and KPZ formula,  \emph{Commun. Math. Phys.}, 2014, Volume 330, Issue 1, pp 283-330.


\bibitem{DuSh} Duplantier B., Sheffield S.: Liouville quantum gravity and KPZ, \emph{Invent. math.}, 185: 333 (2011).


\bibitem{FyHiKe}
Fyodorov Y.V., Hiary G., Keating J.P.: Freezing Transition, Characteristic Polynomials of Random Matrices, and the Riemann Zeta Function, \emph{Phys. Rev. Lett.} {\bf 108}, 170601 (2012).

\bibitem{FyBo}
Fyodorov Y.V., Bouchaud J.P.: Freezing and extreme value statistics in a Random Energy Model with logarithmically correlated potential, \emph{Journal of Physics A: Mathematical and Theoretical}, Volume 41, Number 37, 372001 (2008).

\bibitem{FLeR} Fyodorov Y.V., Le Doussal P., Rosso A.: Statistical Mechanics of Logarithmic REM: Duality, Freezing
and Extreme Value Statistics of $1/f$ Noises generated by Gaussian Free Fields, \emph{J. Stat. Mech.} P10005 (2009).





\bibitem{Disk}
Huang Y., Rhodes R., Vargas V.: Liouville Quantum Gravity on the unit disk, \emph{ Ann. Inst. H. Poincaré Probab. Statist.} 54 (2018), no. 3, 1694--1730. doi:10.1214/17-AIHP852.

 


\bibitem{Kah} Kahane J.-P.: Sur le chaos multiplicatif, \emph{Ann. Sci. Math. Qu{\'e}bec}, \textbf{9}(2), 105-150 (1985).


\bibitem{DOZZ1} Kupiainen A., Rhodes R., Vargas V.: Local conformal structure of Liouville quantum gravity, arXiv:1512.01802.

\bibitem{DOZZ2} Kupiainen A., Rhodes R., Vargas V.: Integrability of Liouville theory: Proof of the DOZZ formula, arXiv:1707.08785.

\bibitem{Mabuchi} Lacoin H., Rhodes R., Vargas V.: Path integral for quantum Mabuchi K-energy, arXiv:1807.01758.

\bibitem{Lambert} Lambert G., Ostrovsky D., Simm N.: Subcritical multiplicative chaos for regularized counting statistics from random matrix theory, \emph{Commun. Math. Phys.} (2018) 360: 1. https://doi.org/10.1007/s00220-018-3130-z.

\bibitem{webb2}  Nikula M., Saksman E., Webb C.: Multiplicative chaos and the characteristic polynomial of the CUE: the L1-phase, arXiv:1806.01831.

\bibitem{Ostro1} Ostrovsky D.: Mellin transform of the limit lognormal distribution, \emph{Commun. Math. Phys.}, \textbf{288}: 287-310 (2009).

\bibitem{Ostro2} Ostrovsky D.: On Barnes beta distributions and applications to the maximum distribution of the 2D Gaussian free field, \emph{J. Stat. Phys.} 164, 1292-1317, (2016).

\bibitem{PaZei}
Paquette E., Zeitouni O.:  The maximum of the CUE field, \emph{International Mathematics Research Notices}, Vol. 2018, No. 16, pp. 5028-5119. 


\bibitem{Pol}
Polyakov A.M.: Quantum geometry of bosonic strings, \emph{Phys. Lett.}, \textbf{103B}: 207-210 (1981).


\bibitem{Powell}
Powell E.: Critical Gaussian chaos: convergence and uniqueness in the derivative normalisation, \emph{Electron. J. Probab.} 23 (2018), paper no. 31, 26 pp. doi:10.1214/18-EJP157.

\bibitem{Annulus} Remy G.: Liouville quantum gravity on the annulus,  	\emph{J. Math. Phys.}, \textbf{59}: 082303 (2018).

\bibitem{RZhu} Remy G., Zhu T.: The distribution of Gaussian multiplicative on the unit interval, arXiv:1804.02942.

\bibitem{review} Rhodes R., Vargas V.: Gaussian multiplicative chaos and applications: a review, \emph{Probability Surveys}, \textbf{11}: 315-392 (2014).

\bibitem{review2} Rhodes R., Vargas V.: Lecture notes on Gaussian multiplicative chaos and Liouville Quantum Gravity, arXiv:1602.07323.

\bibitem{tail} Rhodes R., Vargas V.: The tail expansion of Gaussian multiplicative chaos and the Liouville reflection coefficient, arXiv:1710.02096.


\bibitem{Houches}
Rhodes R., Vargas V.: Lecture notes on Gaussian multiplicative chaos and Liouville Quantum Gravity, to appear in \emph{Les Houches summer school proceedings}, arXiv:1602.07323.

\bibitem{SuZei} Subag E., Zeitouni O.:  Freezing and decorated Poisson point processes, \emph{Commun. Math. Phys.}, Volume 337, Issue 1, pp 55–92, (2015).

\bibitem{ReviewDOZZ} Vargas V.: Lecture notes on Liouville theory and the DOZZ formula, arXiv:1712.00829.

\bibitem{Webb} Webb C.: The characteristic polynomial of a random unitary matrix and Gaussian multiplicative chaos, \emph{Electron. J. Probab.} \textbf{20} (2015), no. 104, 1-21.


\bibitem{ZZ}
Zamolodchikov A.B., Zamolodchikov A.B.: Structure constants and conformal bootstrap in Liouville field theory, \emph{Nuclear Physics B}, \textbf{477} (2), 577-605 (1996).



\end{thebibliography}
\end{document}